\documentclass[11pt,reqno]{amsproc}

\usepackage{txfonts,array,amsmath,amssymb,amsthm,color,tikz,vmargin,overpic,amstext}
\usepackage{amsmath,amsthm,amsfonts,amscd,amssymb,amstext}
\usepackage{mathtools}
\usepackage{booktabs} 
\usepackage[T1]{fontenc}
\usepackage{esint}
\usepackage{commath}
\usepackage{relsize}
\usepackage{graphicx,color}
\usepackage{tabularx}
\usepackage{titletoc}
\usepackage{microtype}
\usepackage{caption}
\usepackage{subcaption}
\usepackage[colorlinks=true,linkcolor=black,citecolor=black,urlcolor=blue]{hyperref}
\usepackage{listings}
\usepackage{cite}
\usepackage{url}
\usepackage{alltt}
\usepackage{tikz}
\usepackage{bbm}
\usepackage{hyperref}
\newtheorem{theorem}{Theorem}

\newtheorem{remark}[theorem]{Remark}

\newtheorem{lemma}[theorem]{Lemma}
\newtheorem{proposition}[theorem]{Proposition}
\newtheorem{corollary}[theorem]{Corollary}

\newtheorem{definition}[theorem]{Definition}
\numberwithin{equation}{section}
\numberwithin{theorem}{section}
\usepackage{suffix}
\usepackage{mathtools}
\usepackage{xparse}
\usepackage[foot]{amsaddr}

\DeclarePairedDelimiterX\MeijerM[3]{\lparen}{\rparen}
{\begin{smallmatrix}#1 \\ #2\end{smallmatrix}\delimsize\vert\,#3}
\newcommand\MeijerG[8][]{%
  G^{\,#2,#3}_{#4,#5}\MeijerM[#1]{#6}{#7}{#8}}

\WithSuffix\newcommand\MeijerG*[7]{%
  G^{\,#1,#2}_{#3,#4}\MeijerM*{#5}{#6}{#7}}

\title{Fluctuations and correlations for products of real asymmetric random matrices}

\author{Will FitzGerald and Nick Simm}
\address{Department of Mathematics, University of Sussex, Brighton, BN1 9RH, United Kingdom}
\email{w.fitzgerald@sussex.ac.uk, n.j.simm@sussex.ac.uk}

\begin{document}

\maketitle
\begin{abstract}
We study the real eigenvalue statistics of products of independent real Ginibre random matrices. These are matrices all of whose entries are real i.i.d. standard Gaussian random variables. For such product ensembles, we demonstrate the asymptotic normality of suitably normalised linear statistics of the real eigenvalues and compute the limiting variance explicitly in both global and mesoscopic regimes. A key part of our proof establishes uniform decorrelation estimates for the related Pfaffian point process, thereby allowing us to exploit weak dependence of the real eigenvalues to give simple and quick proofs of the central limit theorems under quite general conditions. We also establish the universality of these point processes. We compute the asymptotic limit of all correlation functions of the real eigenvalues in the bulk, origin and spectral edge regimes. By a suitable strengthening of the convergence at the edge, we also obtain the limiting fluctuations of the largest real eigenvalue. Near the origin we find new limiting distributions characterising the smallest positive real eigenvalue.
\end{abstract}

\section{Introduction and main results}
For a real random matrix $G$ of size $N \times N$, a basic question of random matrix theory asks simply what is the total number of real eigenvalues of $G$? This question goes back to the work of Edelman, Kostlan and Shub \cite{EKS94} who first computed the expected number of real eigenvalues in the case that the entries of $G$ are all real \textit{i.i.d.} standard Gaussian random variables. This class of random matrices is known as the \textit{real Ginibre ensemble} after Ginibre's 1965 work \cite{Gin65} and is sometimes referred to as GinOE (Ginibre Orthogonal Ensemble) due to its invariance under orthogonal transformations. It is somewhat notorious for being the most technically demanding of the known classical ensembles of Gaussian random matrices, taking more than 40 years to completely understand the structure of its eigenvalue point process \cite{LS91,Sin07, FN07, SW08, BS09}. These works demonstrate an exact solvabality of the GinOE: the eigenvalues (real or complex) form a Pfaffian point process with an explicit correlation kernel. By analysing this kernel, asymptotics of the point process were obtained in the limiting regimes, both in the bulk of the spectrum and near the spectral edges. This led to further investigations showing that the real eigenvalue statistics of real asymmetric random matrices have surprising connections to other fields including connections between the real Ginibre ensemble, annihilating Brownian motions \cite{TZ11} and integrable PDEs \cite{BB20}, and between truncated orthogonal random matrices and Kac polynomials \cite{F10}.

Recently there have been considerable developments pertaining to \textit{products} of random matrices, see for example the survey \cite{AI15}. The main model that we will consider in this paper is a product of independent real Ginibre random matrices. This product ensemble was studied in \cite{IK14, FI16} where it was shown that the eigenvalue correlation functions continue to enjoy the Pfaffian structure known for a single real Ginibre matrix. However, the correlation kernel is considerably more complicated for a general product, being expressed in terms of Meijer G-functions or certain multiple integrals that do not have closed form expressions. For this reason, many basic asymptotic questions regarding products of real Ginibre random matrices remain open. The main notable exception to this concerns the full complex spectrum of the product, for which the techniques of free probability have been successfully applied, see \textit{e.g.} \cite{BJW10, BNS12}. It is less clear that such techniques can be applied to the real eigenvalue statistics considered in this article.

\subsection{Fluctuations of the real eigenvalues}
\label{se:bulkfluct}
In previous works the main available asymptotic results concern the \textit{expected number} of real eigenvalues \cite{FI16,S17}, extending the corresponding result of \cite{EKS94} to products of random matrices. Here we take the next natural step of probability theory and look at the fluctuations. Let $G_{1}, \ldots, G_{m}$ be \textit{i.i.d.} copies of real Ginibre matrices of size\footnote{We assume throughout the article that $N$ is even. We expect similar results to hold when $N$ is odd.} $N \times N$ and consider the product $G^{(m)} = N^{-\frac{m}{2}}\,G_{1}G_{2}\ldots G_{m}$. With this choice of normalisation, the real eigenvalues of $G^{(m)}$ have a limiting empirical spectral distribution supported on the open interval $(-1,1)$ with density given by $\rho(x) = \frac{1}{2m}\,|x|^{\frac{1}{m}-1}\,\mathbbm{1}_{x \in (-1,1)}$, see \cite{FI16,S17}. Let $n=N_{\mathbb{R}}$ denote the total number of real eigenvalues of $G^{(m)}$, and let $\lambda_1,\ldots,\lambda_n$ denote each individual real eigenvalue. Then we define the linear statistic:
\begin{equation}
\xi_{N,m}(f) = \sum_{j=1}^{n}f(\lambda_j). \label{linstat}
\end{equation}
Our first result concerns the limiting distribution of the random variables $\xi_{N,m}(f)$ as $N \to \infty$.
\begin{theorem}
\label{th:ginconv}
Let $f$ be a locally integrable and measurable function satisfying the bound,
\begin{equation}
\sup_{x \in \mathbb{R}}\bigg\{|f(x)|e^{-c|x|^{\frac{2}{m}}}\bigg\} < \infty \label{growthcond}
\end{equation}
for all $c>0$. Then we have the convergence in distribution to a normal random variable,
\begin{equation}
\frac{\xi_{N,m}(f)-\mathbb{E}(\xi_{N,m}(f))}{N^{\frac{1}{4}}} \overset{d}{\longrightarrow} \mathcal{N}(0,\sigma^{2}(f)), \qquad N \to \infty, \label{normalconv}
\end{equation}
with limiting variance
\begin{equation}
\sigma^{2}(f) = \sqrt{\frac{2m}{\pi}}\,(2-\sqrt{2})\,\int_{-1}^{1}dx\,\rho(x)f(x)^{2}. \label{varlim}
\end{equation}
\end{theorem}

In the particular case $f(x) \equiv 1$ the above gives a central limit theorem (CLT) for the number of real eigenvalues of the product matrix $G^{(m)}$. Forrester, Ipsen and Kumar \cite{FIK20} have conjectured such a 
CLT for a product model consisting of truncated orthogonal random matrices, we discuss this model and the corresponding CLT further in Section \ref{se:beyond}. The conditions on $f$ in Theorem \ref{th:ginconv} are quite mild, in particular we do not require compact support, smoothness or even continuity. Previous results comparable to Theorem \ref{th:ginconv} are for a single matrix $m=1$ and in both available cases $f$ is restricted to polynomials \cite{S17-2}, or smooth functions with support strictly contained in $(-1,1)$ \cite{K15}. 

Theorem \ref{th:ginconv} also extends to $N$-dependent test functions. For a fixed $E \in (-1,1)$ and an exponent $\tau > 0$, define
\begin{equation}
\label{mesolinstat}
\xi^{(\tau)}_{N,m}(f) = \sum_{j=1}^{n}f(N^{\tau}(E-\lambda_j)). 
\end{equation}
Then we will refer to $\xi^{(\tau)}_{N,m}(f)$ as a \textit{mesoscopic} linear statistic, as it samples roughly $N^{\frac{1}{2}-\tau}$ points (at least when $E \neq 0$). This regime is intermediate between the global regime, which samples $N^{\frac{1}{2}}$ points, and the local regime which samples $O(1)$ points.
\begin{theorem}
\label{th:ginconvmeso}
Suppose that $f$ is a locally integrable and measurable function satisfying $f \in L^{2}(\mathbb{R})$ and the boundedness $\sup_{x\in \mathbb{R}}\{|(1+|x|)^{1+\delta}f(x)^{2}|\} < \infty$ for some $\delta>0$. Assume $0 < \tau < \frac{1}{2}$ and $E \in (-1,1)\setminus\{0\}$ are fixed. Then we have the convergence in distribution to a normal random variable,
\begin{equation}
\frac{\xi^{(\tau)}_{N,m}(f)-\mathbb{E}(\xi^{(\tau)}_{N,m}(f))}{N^{\frac{1}{4}-\frac{\tau}{2}}} \overset{d}{\longrightarrow} \mathcal{N}(0,\sigma^{2}(f)), \qquad N \to \infty, \label{normalconvmeso}
\end{equation}
with limiting variance
\begin{equation}
\sigma^{2}(f) = \sqrt{\frac{2m}{\pi}}\,(2-\sqrt{2})\,\rho(E)\int_{-\infty}^{\infty}dx\,f(x)^{2}. \label{varlimmeso}
\end{equation}
\end{theorem}
The result \eqref{normalconvmeso} is expected to be sharp in the sense that if $\tau = \frac{1}{2}$ we do not expect convergence to a Gaussian. We will see in Theorem \ref{th:bulkconvintro} below that at scale $\tau = \frac{1}{2}$ the point process of real eigenvalues converges to a universal object related to systems of annihilating Brownian motions. These central limit theorems thus bare some comparison to those proved for related interacting particle systems such as the Arratia flow \cite{GF18}. The constant $2-\sqrt{2}$ in the variance formula \eqref{varlim} was first highlighted in the $m=1$ case in \cite{FN07} and has been referred to as the compressibility of the point process \cite{F15}. In Theorem \ref{th:mesozero} we present a CLT for the case $E=0$ of the linear statistic \eqref{mesolinstat}.
%In Appendix \ref{se:smallest} we present yet another CLT for the case $E=0$ of the linear statistic \eqref{mesolinstat}. 

The proofs of Theorems \ref{th:ginconv} and \ref{th:ginconvmeso} are based on two main ideas. The first is to exploit the Pfaffian structure of the correlation functions obtained in \cite{FI16} and the explicitly known kernel for the point process of real eigenvalues. This allows in principle to write down exact formulas for the cumulants of random variable \eqref{linstat} in terms of integrals involving the test function $f$ and the correlation functions. This approach is well known, see for example \cite{S00}. 

The second idea we use has not been exploited so much in random matrix theory. It is based on the theory of weak dependence, which applies to large classes of point processes whose correlation functions factorise asymptotically at a suitable rate. If this factorisation occurs quickly enough (\textit{e.g.} exponentially fast) then it is possible to deduce limiting Gaussian distributions for statistics of type \eqref{linstat} under certain conditions. Ideas of this type have been known for some time in the physics literature, especially regarding the statistical mechanics of Coulomb systems \cite{MY80}. The theory was further developed and applied with considerable success to point processes and other weakly dependent particle systems \cite{IV82, GF18, NS12,BYY19}, but to our knowledge was not yet applied to Pfaffian point processes or the ensembles considered here. Use of this theory simplifies the estimates of the cumulants and leads to quick and conceptually simple proofs of central limit theorems for real eigenvalue statistics.

\subsection{Kernel asymptotics, correlation decay, and universality}
\label{se:kernasympt}
We begin by recalling the basic structure of the Pfaffian point process that describes the real eigenvalues for products of finite size real Ginibre matrices. Let us say that a $2 \times 2$ matrix kernel $K(x,y)$ is in \textit{derived form} if it can be represented as
\begin{equation}
K(x,y) = \begin{pmatrix} -\frac{\partial}{\partial y}S(x,y) & S(x,y)\\ -S(y,x) & -\int_{x}^{y}dt\,S(t,y)+\frac{1}{2}\,\mathrm{sgn}(x-y) 
\end{pmatrix},\label{derivform}
\end{equation}
for some scalar kernel $S(x,y)$. Therefore the scalar function of two variables $S(x,y)$ completely characterises the kernel and the corresponding point process. The eigenvalue point process of many classical ensembles of random matrix theory take this form. The mentioned GinOE is one example where this structure arises, but also such kernels arise in the analysis of the better known Gaussian Orthogonal Ensemble (GOE) \cite{Meh04}, the real elliptic ensemble \cite{FN08}, truncated orthogonal random matrices \cite{KSZ10} and products thereof \cite{IK14,FIK20}. Certain interacting particle systems also fit into this framework with an additional parameter that corresponds to thinning 
a point process \cite{TZ11,GPTZ18, GTZ20}. Ipsen and Kieburg \cite{IK14}, and Forrester and Ipsen \cite{FI16} have shown that this structure holds for products of GinOE random matrices.
\begin{theorem}[Forrester and Ipsen \cite{FI16}]
\label{th:ik}
Let $G^{(m)} = N^{-\frac{m}{2}}\,G_{1}G_{2}\ldots G_{m}$ be a product of $m$ independent real Ginibre random matrices of size $N \times N$. Then the real eigenvalues of $G^{(m)}$ form a Pfaffian point process with kernel $K_{N}(x,y)$ in the derived form \eqref{derivform} with scalar kernel
\begin{equation}
S_{N}(x,y) = \frac{N^{\frac{3m}{2}}}{(2\sqrt{2\pi})^{m}}\int_{\mathbb{R}}dv\,(x-v)\mathrm{sgn}(y-v)w(N^{\frac{m}{2}}x)w(N^{\frac{m}{2}}v)f_{N-2}(N^{m}xv) \label{prekernel}
\end{equation}
where 
\begin{equation}
w(x) = \int_{\mathbb{R}^{m}}\mathrm{exp}\left(-\frac{1}{2}\sum_{j=1}^{m}\lambda_{j}^{2}\right)\delta(x-\lambda_{1}\ldots \lambda_{m})\,d\lambda_{1}\ldots d\lambda_{m}, \label{ginweightdef1}
\end{equation}
and
\begin{equation}
f_{N-2}(x) = \sum_{j=0}^{N-2}\frac{x^{j}}{(j!)^{m}} .\label{ginfdef1}
\end{equation}
More precisely, given the matrix kernel $K_{N}(x,y)$ constructed as in \eqref{derivform}, the $k^{\mathrm{th}}$ order correlation functions of the real eigenvalues of $G^{(m)}$ are given by the Pfaffian,
\begin{equation}
\rho^{(k)}_{N}(x_1,\ldots,x_k) = \mathrm{Pf}\bigg\{K_{N}(x_i,x_j)\bigg\}_{i,j=1}^{k}.
\end{equation}
\end{theorem}
Given such an exact structure, it is then of interest to calculate the scaled asymptotic behaviour as $N \to \infty$ of the correlation functions in suitable regimes. Our first result of this type gives a strong uniform approximation of the kernel $K_{N}(x,y)$ as $N \to \infty$ in the bulk of the spectrum. It is convenient to define explicitly the scalar kernels in the diagonal entries of the derived form \eqref{derivform},
\begin{align}
D_{N}(x,y) &= -\frac{\partial}{\partial y}S_{N}(x,y) \label{dkernintro},\\
I_{N}(x,y) &= -\int_{x}^{y}dt\,S_{N}(t,y)+\frac{1}{2} \,\mathrm{sgn}(x-y).\label{ikernintro}
\end{align}
We also define the complementary error function,
\begin{equation}
\mathrm{erfc}(z) = \frac{2}{\sqrt{\pi}}\,\int_{z}^{\infty}dt\,e^{-t^{2}}.
\end{equation}
\begin{theorem}[Strong global approximation]
\label{th:globintro}
Given $\epsilon>0$ small, consider the following subset of $[0,1]$,
\begin{equation}
E_{N} = \{x \in [0,1] : \{|x| > N^{-\frac{m}{2}+\epsilon}\} \wedge \{x < 1-N^{-\frac{1}{2}+\epsilon}\}\}. \label{ENsetintro}
\end{equation}
Then the following estimates hold uniformly on $x,y \in E_{N}^{\frac{1}{m}}$ as $N \to \infty$,
\begin{align}
x^{m-1}S_{N}(x^{m},y^{m}) &= \sqrt{\frac{N}{2m\pi}}e^{-\frac{Nm}{2}(x-y)^{2}}(1+o(1)) + O(e^{-N^{\epsilon}}), \label{Snxyintro} \\
(xy)^{m-1}D_{N}(x^{m},y^{m}) &= \frac{N^{\frac{3}{2}}}{m^{\frac{3}{2}}\sqrt{2\pi}}\,m(y-x)\,e^{-\frac{Nm}{2}(x-y)^{2}}(1+o(1))+O(e^{-N^{\epsilon}}), \label{Dnxyintro}\\
I_{N}(x^{m},y^{m}) &= \frac{1}{2}\mathrm{sgn}(x-y)\,\mathrm{erfc}\left(\frac{\sqrt{Nm}(|y-x|)}{\sqrt{2}}\right)(1+o(1)) + O(e^{-N^{\epsilon}}). \label{Inxyintro}
\end{align}
If $x$ and $y$ are negative, and $x,y \in -E_{N}^{\frac{1}{m}}$, the same results apply using the symmetries $S_{N}(-x,-y) = S_{N}(x,y)$, $D_{N}(-x,-y) = -D_{N}(x,y)$ and $I_{N}(-x,-y) = -I_{N}(x,y)$. If $x$ and $y$ have mixed signs, and $|x|,|y| \in E_{N}$ then all three kernels are $O(e^{-N^{\epsilon}})$.
\end{theorem}
The most immediate feature of Theorem \ref{th:globintro} is the exponentially fast decay of the kernel outside the diagonal, at least provided $|x-y| \gg N^{-\frac{1}{2}}$. We will see that this decay leads to a certain clustering property of the real eigenvalues and quantifies their weakly dependent structure, for more precise statements see Section \ref{se:clustering}. In general, eigenvalues of random matrices are strongly correlated random variables, so the relatively weak correlations of purely real eigenvalues may not be immediately obvious. This becomes much clearer in light of Theorem \ref{th:globintro}.
%. 
%\textcolor{red}{We emphasize that the structure of weak dependence only appears asymptotically as $N \rightarrow \infty$, as for finite $N$ the correlation functions of the real eigenvalues do not decay exponentially fast. }\textcolor{blue}{Is this obvious? Even at finite-$N$ it looks like the correlations decay exponentially? As they are essentially a finite degree polynomial multiplied by a Gaussian $e^{-x^{2}/2-y^{2}/2}$.}
%We emphasize that this structure only appears asymptotically as $N \to \infty$, as for finite-$N$ the real eigenvalues are correlated random variables. 

Another consequence of Theorem \ref{th:globintro} is the bulk convergence of the correlation kernel. To this end, let us take a point $E \in (-1,1)\setminus\{0\}$ fixed and consider the bulk scaling,
\begin{equation}
s_{N,m}^{(\mathrm{bulk})}(\xi,\zeta) = \frac{1}{2\sqrt{Nm}\rho(E)}\,S_{N}\left(E+\frac{\xi}{2\sqrt{Nm}\rho(E)},E+\frac{\zeta}{2\sqrt{Nm}\rho(E)}\right), \label{scaledKintro}
\end{equation}
where $\rho(E) = \frac{1}{2m}\,|E|^{\frac{1}{m}-1}$ is the limiting density of eigenvalues. We denote by $k^{(\mathrm{bulk})}_{N,m}(\xi,\zeta)$ the $2 \times 2$ matrix kernel in derived form with scalar kernel $s_{N,m}^{(\mathrm{bulk})}(\xi,\zeta)$.

\begin{theorem}
\label{th:bulkconvintro}
We have the convergence to a limiting kernel $k^{(\mathrm{bulk})}_{N,m}(\xi,\zeta) \to k_{\infty}^{(\mathrm{bulk})}(\xi,\zeta)$ uniformly in compact subsets of $\xi$ and $\zeta$, where the limiting $2 \times 2$ matrix kernel $k_{\infty}^{(\mathrm{bulk})}(\xi,\zeta)$ is in derived form with scalar kernel,
\begin{equation}
s^{(\mathrm{bulk})}_{\infty}(\xi,\zeta) = \frac{1}{\sqrt{2\pi}}\,e^{-\frac{1}{2}\,(\xi-\zeta)^{2}}.
\end{equation} 
\end{theorem}
The above result was previously obtained in \cite{BS09} for $m=1$, \textit{i.e.} for the real eigenvalues of a single real Ginibre random matrix in the bulk. One might therefore anticipate Theorem \ref{th:bulkconvintro} from the general universality principles of random matrix theory. A similar type of universality for products of complex Ginibre random matrices was obtained in the works \cite{DDZ16,DY16}. The limiting kernel $k_{\infty}^{(\mathrm{bulk})}(\xi,\zeta)$ also arises outside of random matrix theory; it was shown in \cite{TZ11} that the Pfaffian point process with kernel $k_{\infty}^{(\mathrm{bulk})}(\xi,\zeta)$ is equivalent to the point process of annihilating Brownian motions under the maximal entrance law. 

The other interesting scaling limit to consider for such correlation kernels lies in a small neighbourhood of the spectral edge. Without loss of generality we will focus on the right end point of the spectrum located at $x=1$. We define
\begin{align}
s^{(\mathrm{edge})}_{N,m}(\xi,\zeta) &= \frac{1}{2\sqrt{Nm}\rho(1)}\,S_{N}\left(1+\frac{\xi}{2\sqrt{Nm}\rho(1)},1+\frac{\zeta}{2\sqrt{Nm}\rho(1)}\right).
\end{align}
Let $k^{(\mathrm{edge})}_{N,m}(\xi,\zeta)$ be the $2 \times 2$ matrix kernel in derived form with scalar kernel $s_{N,m}^{(\mathrm{edge})}(\xi,\zeta)$.
\begin{theorem}
\label{th:edgeuniformityintro}
For a fixed $s \in \mathbb{R}$ consider the set $A = \{ \xi \in \mathbb{R} : \xi > s\}$. Then we have the convergence $k_{N,m}^{(\mathrm{edge})}(\xi,\zeta) \to k_{\infty}^{(\mathrm{edge})}(\xi,\zeta)$ as $N \to \infty$ uniformly for $(\xi,\zeta) \in A^{2}$, where the limiting $2 \times 2$ matrix kernel $k^{(\mathrm{edge})}_{\infty}(\xi,\zeta)$ is in derived form with scalar kernel,
\begin{equation}
s_{\infty}^{(\mathrm{edge})}(\xi,\zeta) = \frac{1}{2\sqrt{2\pi}}\,e^{-\frac{1}{2}(\xi-\zeta)^{2}}\mathrm{erfc}\left(\frac{\xi+\zeta}{\sqrt{2}}\right)+\frac{1}{4\sqrt{\pi}}\,e^{-\xi^{2}}\mathrm{erfc}(-\zeta).
\end{equation}
\end{theorem}
The limit kernel $k_{\infty}^{(\mathrm{edge})}(\xi,\zeta)$ first arose in the $m=1$ real Ginibre case \cite{BS09} (see \cite{BPSTZ16} for a corrected version). As with the bulk regime it also describes systems of annihilating Brownian motions, now with a half-space initial condition \cite{GPTZ18}. The convergence of Theorem \ref{th:edgeuniformityintro} is a relatively strong one as it holds uniformly on sets of unbounded height, even for $m=1$ it is stronger than the previously known convergence of \cite{BS09}. Such a uniformity will be helpful later in obtaining distributional convergence of the largest real eigenvalue. 

Near the origin, in contrast, there are new universality classes which are much more sensitive to the number of factors $m$ in the product. Define the scaled kernel near the origin,
\begin{equation}
s^{(\mathrm{origin})}_{N,m}(\xi,\zeta) = \frac{1}{N^{\frac{m}{2}}}\,S_{N}\left(\frac{\xi}{N^{\frac{m}{2}}},\frac{\zeta}{N^{\frac{m}{2}}}\right).
\end{equation}
As before, we denote by $k^{(\mathrm{origin})}_{N,m}(\xi,\zeta)$ the $2 \times 2$ matrix kernel in derived form with scalar kernel $s^{(\mathrm{origin})}_{N,m}(\xi,\zeta)$.
\begin{theorem}
\label{th:origkern}
We have the convergence to a limiting kernel $k^{(\mathrm{origin})}_{N,m}(\xi,\zeta) \to k^{(\mathrm{origin})}_{\infty,m}(\xi,\zeta)$ uniformly on compact subsets of $\xi$ and $\zeta$, where the limiting $2 \times 2$ matrix kernel $k^{(\mathrm{origin})}_{\infty,m}(\xi,\zeta)$ is in derived form with scalar kernel,
\begin{equation}
s^{(\mathrm{origin})}_{\infty,m}(\xi,\zeta) = \frac{1}{(2\sqrt{2\pi})^{m}}\,\int_{\mathbb{R}}d\eta\,(\xi-\eta)\mathrm{sgn}(\zeta-\eta)w(\xi)w(\eta)f_{\infty}(\xi\zeta),
\end{equation}
where the weight $w$ is defined in \eqref{ginweightdef1} and
\begin{equation}
f_{\infty}(\xi\zeta) := \sum_{k=0}^{\infty}\frac{(\xi\zeta)^{k}}{(k!)^{m}}.
\end{equation}
\end{theorem}
\begin{remark}
The above finite-$N$ matrix kernels $k^{(\mathrm{bulk})}_{N,m}(\xi,\zeta), k^{(\mathrm{edge})}_{N,m}(\xi,\zeta)$ and $k^{(\mathrm{origin})}_{N,m}(\xi,\zeta)$ are not precisely those obtained by the corresponding rescaling of $K_{N}(x,y)$, as the latter contains slightly different normalisation factors in the diagonal entries of its derived form. However, simple properties of Pfaffians show that these pre-factors cancel out and correspond to the same correlation functions. In the literature such kernels are said to be conjugation equivalent or gauge equivalent.
\end{remark} 
\subsection{Fluctuations of extreme eigenvalues}
\label{se:fluctext}
The previous asymptotics for the correlation kernel can be used to extract convergence in distribution style results for the extreme eigenvalues, both for small eigenvalues near the origin and the maximal eigenvalue at the edge. Before we state these results we briefly review the known case $m=1$ that has been the subject of recent interest. 

The largest real eigenvalue in the GinOE was first described in the work of Rider and Sinclair \cite{RS14} who computed its distribution in terms of certain Fredholm Pfaffians at finite-$N$. Concretely, these are series expansions for the cumulative distribution function of the largest real eigenvalue where each term in the series is built from the GinOE correlation functions. Then passing to the limit $N \to \infty$ with appropriate centering and scaling, one finds the limiting distribution in terms of a certain infinite dimensional operator determinant  (the original expression in \cite{RS14} was later corrected, see Theorem 1.1 of \cite{PTZ17}). The kernel of the limiting operator has a convolution form similar to the kernel that appears in the Tracy-Widom formula for the largest eigenvalue of the GOE \cite{TW94,TW96}.
In particular it belongs to a class studied using probabilistic techniques in \cite{FTZ20, FTZ21} and this gives one route to obtain the tail asymptotics of the largest real eigenvalue. It has been shown that the mentioned determinant can be re-formulated using the theory of integrable systems. In the GOE case the largest eigenvalue is famously related to a solution of the Painlev\'e II equation, whereas in the GinOE case it was very recently shown to be related to a solution of the Zakharov-Shabat system \cite{BB20}. These quantities can be characterised as solutions of an appropriately defined Riemann-Hilbert problem. For the GinOE, applying the non-linear steepest descent method to the Riemann-Hilbert problem allowed the authors of \cite{BB20} to obtain tail asymptotics of the largest real eigenvalue. Very recently, the Zakharov-Shabat system has appeared in the study of the large deviations of the Kardar-Parisi-Zhang equation with weak noise \cite{LDK}. 

Our next result shows that products of GinOE random matrices have largest real eigenvalue belonging to the mentioned Zakharov-Shabat universality class. Again let $G_{1}, \ldots, G_{m}$ be \textit{i.i.d.} copies of real Ginibre matrices of size $N \times N$ and consider the product $G^{(m)} = N^{-\frac{m}{2}}\,G_{1}G_{2}\ldots G_{m}$.
\begin{theorem}
\label{th:lambdamax}
Let $\lambda^{(m)}_{N,\mathrm{max}}$ denote the largest real eigenvalue of the product matrix $G^{(m)}$. Then we have the convergence in distribution,
\begin{equation}
\sqrt{\frac{N}{m}}\,\left(\lambda^{(m)}_{N,\mathrm{max}}-1\right) \overset{d}{\longrightarrow}\lambda_{\mathrm{max}}, \qquad N \to \infty, \label{convreal}
\end{equation}
where the limiting random variable $\lambda_{\mathrm{max}}$ is independent of $m$ and has distribution given by the following absolutely convergent series,
\begin{equation}
\mathbb{P}(\lambda_{\mathrm{max}} < s) = \sum_{\ell=1}^{\infty}\frac{(-1)^{\ell}}{\ell!}\,\int_{[s,\infty)^{\ell}}\prod_{j=1}^{\ell}d\xi_{j}\,\mathrm{Pf}\bigg\{k^{(\mathrm{edge})}_{\infty}(\xi_{i},\xi_{j})\bigg\}_{i,j=1}^{\ell}.
\end{equation}
\end{theorem} 
Since the limiting random variable of Theorem \ref{th:lambdamax} is independent of $m$, it coincides with the $m=1$ characterisations given in \cite{BB20}. The notable $m$ dependence lies in the normalisation of \eqref{convreal} which suggests that fluctuations of the largest real eigenvalue become larger for increasing $m$.

On the other hand, the smallest eigenvalues have limiting distributions with a more complicated dependence on the number of factors $m$ in the product. Let $\lambda^{(m)}_{\mathrm{min},N}$ denote the smallest positive real eigenvalue of the product matrix $G^{(m)}$. It is worth noting that this random variable is not obviously related to the singular values of $G^{(m)}$, as the latter would be influenced also by neighbouring complex eigenvalues. After blowing up by a factor $N^{\frac{m}{2}}$, the random variable $\lambda^{(m)}_{\mathrm{min},N}$ has a limiting distribution.
\begin{theorem}
\label{th:smalleigintro}
We have the convergence in distribution
\begin{equation}
N^{\frac{m}{2}}\lambda^{(m)}_{\mathrm{min},N} \overset{d}{\longrightarrow} \lambda^{(m)}_{\mathrm{min}}, \qquad N \to \infty,
\end{equation}
where the limiting random variable has distribution given by the following absolutely convergent series,
\begin{equation}
\mathbb{P}(\lambda^{(m)}_{\mathrm{min}} > s) = \sum_{\ell=1}^{\infty}\frac{(-1)^{\ell}}{\ell!}\,\int_{[-s,s]^{\ell}}\prod_{j=1}^{\ell}d\xi_{j}\,\mathrm{Pf}\bigg\{k^{(\mathrm{origin})}_{\infty,m}(\xi_{i},\xi_{j})\bigg\}_{i,j=1}^{\ell}. \label{mindist}
\end{equation}
\end{theorem}
In comparison with the largest real eigenvalue, it would be interesting to understand whether the limiting distributions \eqref{mindist} can be explored using probabilistic \cite{FTZ20,FTZ21} or integrable approaches \cite{BB20}, for example to obtain characterisations of \eqref{mindist} in terms of differential equations, or to extract the tail asymptotics as $s \to \infty$. 

\subsection{Beyond real Ginibre random matrices}
\label{se:beyond}
The methods developed in this paper likely apply to various other models of real asymmetric random matrices. One particular example is a product of truncated Haar distributed orthogonal random matrices, which has been the subject of recent investigations \cite{FK18, FIK20, LMS21}. As in the case of real Ginibre random matrices, eigenvalue statistics of such products are again described by a Pfaffian point process. This was shown in the single matrix case in the work \cite{KSZ10}, then for products of such matrices beginning with \cite{IK14} and simplified expressions for the kernel directly analogous to the one of Theorem \ref{th:ik} obtained in \cite{FIK20}. Then it is possible to obtain analogues of the main results discussed here in Sections \ref{se:bulkfluct}, \ref{se:kernasympt} and \ref{se:fluctext} for such products employing a similar approach developed in this paper. The main difference in the statement of results is that the limiting density $\rho(x) = \frac{1}{2m}\,|x|^{-1+\frac{1}{m}}\mathbbm{1}_{x \in (-1,1)}$ that appeared earlier should be replaced with the density
\begin{equation}
\rho(x) = \frac{1}{2m}\,|x|^{-1+\frac{1}{m}}\,\frac{1}{1-|x|^{\frac{2}{m}}}\mathbbm{1}_{x \in (-\tilde{\alpha}^{\frac{m}{2}}, \tilde{\alpha}^{\frac{m}{2}})}, \label{rhoT}
\end{equation}
where $\tilde{\alpha} := \lim_{N \to \infty}\left(1+\frac{L_{N}}{N}\right)^{-1}$; here  $L_{N}$ is the number of rows and columns truncated, and $N$ is the dimension of the truncated matrix. In particular, this replacement gives the appropriate analogue of Theorem \ref{th:ginconv}, conjectured in the case $f(x) \equiv 1$ in \cite{FIK20}.

Another model well suited to the methods described in this article is known as the real elliptic ensemble. This model describes an interpolation between GOE and GinOE random matrices and is constructed as follows. Let $S$ and $A$ be $N \times N$ random matrices sampled from the GOE and anti-symmetric GOE respectively.
Then $X^{(\tau)} = \sqrt{\frac{1+\tau}{2N}} S + \sqrt{\frac{1-\tau}{2N}} A$ is an $N \times N$ random matrix sampled from the real elliptic ensemble with interpolation parameter 
$\tau \in [0, 1]$. The Pfaffian structure of the eigenvalue point process and explicit correlation kernel in derived form was obtained in \cite{FN08}. The work of \cite{FSK98} on the complex elliptic ensemble also provides an
expression for the scalar kernel $S_{N}(x,y)$ in the real elliptic ensemble in terms of incomplete Gamma functions that is convenient for asymptotics; and \cite{AKV16} obtain asymptotics of such kernels with uniform error bounds. Combining these results allows one to obtain a global approximation of the type given in Theorem \ref{th:globintro} for this ensemble. We believe that following a similar strategy outlined in the present paper would extend those results to a CLT similar to Theorem \ref{th:ginconv} with $0 < \tau < 1$ fixed (called the strongly non-Hermitian regime). Estimates on the expectation and variance for the number of real eigenvalues in the real elliptic ensemble were recently investigated in the weakly non-Hermitian regime \cite{BKLL21}.\\

This paper is structured as follows. In order to keep the paper self-contained and of appropriate length we will focus on products of real Ginibre random matrices. We begin in Section \ref{se:prelims} by obtaining some alternative representations for the kernel $K_{N}(x,y)$ of Theorem \ref{th:ik} and use these to give the proof of Theorem \ref{th:globintro} and its consequence, Theorem \ref{th:bulkconvintro}. Based on Theorem \ref{th:globintro}, in Section \ref{se:variance} we obtain asymptotic formulas for the variance of linear statistics of the real eigenvalues. Then we develop the theory of weak dependence as it applies to real Ginibre matrices and their products in Sections \ref{se:cumulants} and \ref{se:clustering}. This is used to complete the proof of Theorems \ref{th:ginconv} and \ref{th:ginconvmeso} in Section \ref{se:cumulantbound}, including a CLT for the case $E=0$ of \eqref{mesolinstat}. Section \ref{se:edge} studies the edge statistics and contains the proofs of Theorems \ref{th:edgeuniformityintro} and \ref{th:lambdamax}. The Appendix contains the proof of miscellaneous results, including the proof of Theorems \ref{th:origkern} and \ref{th:smalleigintro}.

\section*{Acknowledgements}
Both authors gratefully acknowledge financial support of the Royal Society, grants URF\textbackslash R1\textbackslash 180707 and RGF\textbackslash EA\textbackslash 181085.

\section{Preliminaries and proof of the global approximation}
\label{se:prelims}
The goal of this Section is to prove Theorem \ref{th:globintro}. In order to extract the asymptotic behaviour of the kernel $K_{N}(x,y)$ defined in Theorem \ref{th:ik} we begin by obtaining some alternative expressions for the various scalar kernels comprising $K_{N}(x,y)$, namely the functions $S_{N}(x,y)$, $D_{N}(x,y)$ and $I_{N}(x,y)$ mentioned in Theorem \ref{th:globintro} and \eqref{dkernintro}, \eqref{ikernintro}. These alternative forms turn out to be better suited for performing asymptotic analysis; the reasons for this will be discussed later in Remark \ref{rem:saddle}. Then we will need asymptotic results for the main quantities appearing in the integrand of \eqref{prekernel}, namely the weight \eqref{ginweightdef1} and the sum \eqref{ginfdef1}. Bringing these results together will then allow us to complete the proof of Theorem \ref{th:globintro}.
\begin{lemma}
\label{lem:kernrep}
Define the pre-factors,
\begin{equation}
\begin{split}
C_{N,m} &= \frac{N^{\frac{3m}{2}}}{(2\sqrt{2\pi})^{m}},\label{cnmdef}\\
D_{N,m} &= N^{\frac{m(N-3)}{2}}2^{m\left(\frac{N-1}{2}\right)}\left(\frac{\Gamma\left(\frac{N-1}{2}\right)}{(N-2)!}\right)^{m}.
\end{split}
\end{equation}
Then we have the following equivalent representations for the scalar kernel $S_{N}(x,y)$ in \eqref{prekernel},
\begin{equation}
\begin{split}
S_{N}(x,y) = &-2C_{N,m}\int_{y}^{\infty}dv\,(x-v)w(N^{\frac{m}{2}}x)w(N^{\frac{m}{2}}v)f_{N-2}(N^{m}xv)\\
&+C_{N,m}D_{N,m}x^{N-1}w(N^{\frac{m}{2}}x),\label{Srep1}
\end{split}
\end{equation}
\begin{equation}
\begin{split}
S_{N}(x,y) = &2C_{N,m}\int_{-\infty}^{y}dv\,(x-v)w(N^{\frac{m}{2}}x)w(N^{\frac{m}{2}}v)f_{N-2}(N^{m}xv)\\
&-C_{N,m}D_{N,m}x^{N-1}w(N^{\frac{m}{2}}x), \label{Srep2}
\end{split}
\end{equation}
and
\begin{equation}
\begin{split}
S_{N}(x,y) = &2C_{N,m}\int_{0}^{y}dv\,(x-v)w(N^{\frac{m}{2}}x)w(N^{\frac{m}{2}}v)f_{N-2}(N^{m}xv)\\
&-C_{N,m}\left(\frac{2}{N}\right)^{m}w(N^{\frac{m}{2}}x).\label{Srep3}
\end{split}
\end{equation}
\end{lemma}

\begin{proof}
From the definition of the weight \eqref{ginweightdef1} it is straightforward to check that,
\begin{equation}
\int_{0}^{\infty}dv\,v^{k}w(N^{\frac{m}{2}}v) = \frac{1}{2}\,\left(\frac{2}{N}\right)^{\frac{m(k+1)}{2}}\left(\Gamma\left(\frac{k+1}{2}\right)\right)^{m}.\label{eqmomweight}
\end{equation}
Then by inserting the definition of $f_{N-2}$ from \eqref{ginfdef1} and integrating term by term we obtain the pair of exact identities:
\begin{align}
&\int_{-\infty}^{\infty}dv\,\mathrm{sgn}(v)\,(v-x)w(N^{\frac{m}{2}}v)f_{N-2}(N^{m}xv) = \left(\frac{2}{N}\right)^{m},\label{intid1}\\
&\int_{-\infty}^{\infty}dv\,(x-v)w(N^{\frac{m}{2}}v)f_{N-2}(N^{m}xv) = D_{N,m}x^{N-1}. \label{intid2}
\end{align}
To see the first identity \eqref{intid1} note that
\begin{align}
&\int_{0}^{\infty}dv\,(v-x)w(N^{\frac{m}{2}}v)f_{N-2}(N^{m}xv)-\int_{-\infty}^{0}dv\,(v-x)w(N^{\frac{m}{2}}v)f_{N-2}(N^{m}xv)\\
%&=\int_{0}^{\infty}dv\,(v-x)w(N^{\frac{m}{2}}v)f_{N-2}(N^{m}xv)+\int_{0}^{\infty}dv\,(v+x)w(N^{\frac{m}{2}}v)f_{N-2}(-N^{m}xv)\\
&=2\int_{0}^{\infty}dv\,vw(N^{\frac{m}{2}}v)f^{(e)}_{N-2}(N^{m}xv)-2\int_{0}^{\infty}dv\,xw(N^{\frac{m}{2}}v)f^{(o)}_{N-2}(N^{m}xv) \label{fefo}
\end{align}
where $f_{N-2}^{(e)}$ and $f_{N-2}^{(o)}$ denote the even and odd monomial contributions to $f_{N-2}$,
\begin{equation}
 f^{(e)}_{N-2}(v) = \sum_{j=0}^{\frac{N}{2}-1}\frac{v^{2j}}{((2j)!)^{m}}, \qquad
 f^{(o)}_{N-2}(v) = \sum_{j=0}^{\frac{N}{2}-2}\frac{v^{2j+1}}{((2j+1)!)^{m}}.
\end{equation}
Explicitly computing the integrals in \eqref{fefo} using \eqref{eqmomweight} results in a cancellation between adjacent terms in the sum, so that only the $j=0$ term from $f^{(e)}_{N-2}$ gives a non-zero contribution. Again by \eqref{eqmomweight}  this gives $(2/N)^{m}$. The proof of \eqref{intid2} is similar, except terms with the same index cancel and one is left with only the $j=\frac{N}{2}-1$ term in $f^{(e)}_{N-2}$.

Now starting from \eqref{prekernel} and rearranging the integration limits gives
\begin{equation}
\begin{split}
&S_{N}(x,y) = C_{N,m}\int_{-\infty}^{\infty}dv\,(x-v)w(N^{\frac{m}{2}}v)w(N^{\frac{m}{2}}x)f_{N-2}(N^{m}xv)\\
& - 2C_{N,m}\int_{y}^{\infty}dv\,(x-v)w(N^{\frac{m}{2}}v)w(N^{\frac{m}{2}}x)f_{N-2}(N^{m}xv). \label{trivrewrite}\\
\end{split}
\end{equation}
Inserting \eqref{intid2} into \eqref{trivrewrite} completes the proof of representation \eqref{Srep1}. 
%For \eqref{Srep2} we instead insert \eqref{intid2} into
%\begin{equation}
%\begin{split}
%&S_{N}(x,y) = -C_{N,m}\int_{-\infty}^{\infty}dv\,(x-v)w(N^{\frac{m}{2}}v)w(N^{\frac{m}{2}}x)f_{N-2}(N^{m}xv)\\
%& + 2C_{N,m}\int_{-\infty}^{y}dv\,(x-v)w(N^{\frac{m}{2}}v)w(N^{\frac{m}{2}}x)f_{N-2}(N^{m}xv). \label{trivrewrite2}\\
%\end{split}
%\end{equation}
We obtain \eqref{Srep2} in a similar manner. Finally, to obtain \eqref{Srep3} we insert \eqref{intid1} into
\begin{equation}
\begin{split}
S_{N}(x,y) &= 2C_{N,m}\int_{0}^{y}dv\, (x-v)w(N^{\frac{m}{2}}v)w(N^{\frac{m}{2}}x)f_{N-2}(N^{m}xv)\\
&-C_{N,m}\int_{-\infty}^{\infty}dv\,\mathrm{sgn}(v)(x-v)w(N^{\frac{m}{2}}v)w(N^{\frac{m}{2}}x)f_{N-2}(N^{m}xv).
\end{split}
\end{equation}
\end{proof}
Next we consider the kernels $D_{N}(x,y)$ and $I_{N}(x,y)$ defined in \eqref{dkernintro} and \eqref{ikernintro}. 
\begin{lemma}
\label{lem:dandi}
The kernels $D_{N}(x,y)$ and $I_{N}(x,y)$ are anti-symmetric functions of $x$ and $y$. Furthermore, the kernel $D_{N}(x,y)$ is given by
\begin{equation}
D_{N}(x,y) = 2C_{N,m}(x-y)w(N^{\frac{m}{2}}x)w(N^{\frac{m}{2}}y)f_{N-2}(N^{m}xy). \label{drep2}
\end{equation}
When $y>x$ the kernel $I_{N}(x,y)$ has the alternative expression,
\begin{equation}
I_{N}(x,y) = \int_{0}^{x}dt\,S_{N}(t,y)+C_{N,m}\left(\frac{2}{N}\right)^{m}\int_{0}^{y}dt\,w(N^{\frac{m}{2}}t)-\frac{1}{2}. \label{irep2}
\end{equation}
\end{lemma}

\begin{proof}
The anti-symmetry of $D_{N}$ and $I_{N}$ is implicit in the construction of \cite{FI16}, but we give a direct proof here. The representation \eqref{drep2} follows by inserting \eqref{Srep1} into \eqref{dkernintro}. This kernel is clearly anti-symmetric. To see the anti-symmetry of $I_{N}(x,y)$ it suffices to check anti-symmetry of the first term in \eqref{ikernintro}. To that end, by \eqref{Srep1} we have
\begin{equation}
S_{N}(t,y)-S_{N}(t,x) = -2C_{N,m}\int_{x}^{y}dv\,(v-t)w(N^{\frac{m}{2}}v)w(N^{\frac{m}{2}}t)f_{N-2}(N^{m}vt)
\end{equation}
and thus by symmetry in the $v$ and $t$ variables,
\begin{equation}
\int_{x}^{y}dt\,(S_{N}(t,y)-S_{N}(t,x)) = 0,
\end{equation}
which implies the anti-symmetry of the first term in \eqref{ikernintro}. Hence $I_{N}(x,y) = -I_{N}(y,x)$. To establish \eqref{irep2}, using \eqref{ikernintro} we decompose,
\begin{equation}
\begin{split}
I_{N}(x,y) &= \int_{0}^{x}dt\,S_{N}(t,y)-\int_{0}^{y}dt\,S_{N}(t,y)-\frac{1}{2}\\
&= \int_{0}^{x}dt\,S_{N}(t,y)+C_{N,m}\left(\frac{2}{N}\right)^{m}\int_{0}^{y}dt\,w(N^{\frac{m}{2}}t)-\frac{1}{2} \label{symm0y}
\end{split}
\end{equation}
where to obtain the second line \eqref{symm0y} we used \eqref{Srep3} and again used the symmetry of the integrand in $v$ and $t$.
\end{proof}
We pause here to recollect the standard asymptotic results for the weights \eqref{ginweightdef1} and the sum \eqref{ginfdef1}. These quantities all have convenient integral representations, and the following three propositions follow from applying the Laplace method of asymptotics to those integral representations. We do not give the details of the proof here, but instead refer the author to Appendix A of \cite{LMS21} where the main ideas are discussed. These asymptotics were also considered in the context of complex Ginibre random matrices in \cite{AB12}.
\begin{remark}
Throughout the paper $C$ and $c$ will always denote absolute positive constants that are uniform in the sense that they are independent of any relevant asymptotic parameters, function arguments or integration variables. Their precise value will be considered unimportant and may change from line to line.
\end{remark}
\begin{proposition}
\label{prop:gin1}
Fix a large constant $M>0$. Then we have the following asymptotic estimate uniformly on $|x| \in [MN^{-\frac{m}{2}},\infty)$
\begin{equation}
w(N^{m/2}x) = N^{-\frac{m-1}{2}}\,e^{-\frac{Nm}{2}x^{\frac{2}{m}}}\frac{(4\pi)^{\frac{m-1}{2}}}{\sqrt{m}}|x|^{-\frac{m-1}{m}}\left(1+O\left(\frac{1}{N|x|^{\frac{2}{m}}}\right)\right), \qquad N \to \infty. \label{wresult}
\end{equation}
Furthermore, we have the crude bound uniformly on $[MN^{-\frac{m}{2}},\infty)$,
\begin{equation}
\label{crudegin}
w(N^{m/2}x) \leq C\,e^{-\frac{Nm}{2}x^{\frac{2}{m}}}.
\end{equation}
\end{proposition}
\begin{proposition}
\label{prop:gin2}
As $N \to \infty$ we have the following estimate uniformly on $x \in \mathbb{R}\setminus((1-\omega)^{m},(1+\omega)^{m})$,
\begin{equation}
f_{N-2}(N^{m}x) = f_{\infty}(N^{m}x)\mathbbm{1}_{-(1+\omega)^{m} < x < (1-\omega)^{m}} + \frac{x^{N-1}e^{mN}}{(2\pi N)^{m/2}(x-1)}\left(1+O\left(\frac{1}{\sqrt{N}\omega}\right)\right), \label{fintest}
\end{equation}
where we take $\omega = N^{-\frac{1}{2}}$.
\end{proposition}

\begin{proposition}
\label{prop:gin3}
Fix a large constant $M>0$. Then we have the following estimate uniformly on $x \in (MN^{-m},\infty)$,
\begin{equation}
f_{\infty}(N^{m}x) = (2\pi)^{-(m-1)/2}x^{-\frac{m-1}{2m}}e^{Nmx^{\frac{1}{m}}}N^{-\frac{m-1}{2}}\,\frac{1}{\sqrt{m}}\left(1+O\left(\frac{1}{x^{\frac{1}{m}}N}\right)\right), \qquad N \to \infty. \label{asyfint}
\end{equation}
We have the crude bound, for $x > MN^{-m}$, 
\begin{equation}
f_{\infty}(N^{m}x) \leq C\,e^{Nmx^{\frac{1}{m}}}. \label{crudefn}
\end{equation}
\end{proposition}

Using these asymptotics we can further simplify the representation \eqref{irep2}.
\begin{lemma}
Let $0 < x < y$. Then for any $\epsilon, \epsilon'$ such that $0 < \epsilon' < \epsilon$ we have uniformly on $\{x > N^{-\frac{m}{2}+\epsilon}\}$, 
\begin{equation}
I_{N}(x,y) = \int_{N^{-\frac{m}{2}+\epsilon'}}^{x}dt\,S_{N}(t,y) + O(e^{-N^{\epsilon}}), \qquad N \to \infty. \label{irep3}
\end{equation}
\end{lemma}
\begin{proof}
We first note that
\begin{equation}
I_{N}(x,y) = \int_{0}^{x}dt\,S_{N}(t,y) + O(e^{-cN^{\epsilon}}).
\end{equation}
To see this, we use the exact identity
\begin{equation}
C_{N,m}\left(\frac{2}{N}\right)^{m}\int_{0}^{\infty}dt\,w(N^{\frac{m}{2}}t) = \frac{1}{2},
\end{equation}
which follows from \eqref{eqmomweight}. Comparing with the middle term of \eqref{irep2} the error in this approximation is,
\begin{equation}
C_{N,m}\left(\frac{2}{N}\right)^{m}\int_{y}^{\infty}dt\,w(N^{\frac{m}{2}}t) = O(e^{-cN^{\epsilon}}), \label{weightint}
\end{equation}
where we used that $y > N^{-\frac{m}{2}+\epsilon}$.

It suffices to show that choosing $0 < \epsilon' < \epsilon$ we have
\begin{equation}
\int_{0}^{MN^{-\frac{m}{2}+\epsilon'}}dt\,S_{N}(t,y) = O(e^{-cN^{\epsilon}}) \label{suff}
\end{equation}
Inserting \eqref{Srep1} we note that the contribution to \eqref{suff} from the second term of \eqref{Srep1} is exponentially small, this follows from Stirling's formula and the presence of the term $t^{N-1}$. Next we consider the contribution to \eqref{suff} from first term of \eqref{Srep1}. We apply the bound $f_{N-2}(tvN^{m}) < f_{\infty}(tvN^{m})$ and split the integration region by first supposing that $v > v^{*} := \mathrm{max}(MN^{-m}/t,N^{-\frac{m}{2}+\epsilon})$ for a large constant $M$, so that $tvN^{m} > M$ and we may use the crude bounds \eqref{crudefn} and \eqref{crudegin} to obtain,
\begin{align}
&w(N^{\frac{m}{2}}t)w(N^{\frac{m}{2}}v)f_{\infty}(tvN^{m}) \leq w(N^{\frac{m}{2}}t)e^{\frac{m}{2}Nt^{\frac{2}{m}}}e^{-\frac{m}{2}N(t^{\frac{1}{m}}-v^{\frac{1}{m}})^{2}}\\
&= e^{-c(N-1)(t^{\frac{1}{m}}-v^{\frac{1}{m}})^{2}}
w(N^{\frac{m}{2}}t)e^{\frac{m}{2}Nt^{\frac{2}{m}}}e^{-\frac{m}{2}(t^{\frac{1}{m}}-v^{\frac{1}{m}})^{2}} \label{crudeinbnd2}\\
&\leq Ce^{-cN^{\frac{2\epsilon}{m}}}w(N^{\frac{m}{2}}t)e^{\frac{mN}{2}t^{\frac{2}{m}}}e^{-\frac{m}{2}(t^{\frac{1}{m}}-v^{\frac{1}{m}})^{2}}
\label{crudeinbnd3}
\end{align}
where we used that $0 < t < N^{-\frac{m}{2}+\epsilon'}$ and $v > N^{-\frac{m}{2}+\epsilon}$ with $\epsilon > \epsilon'$ in \eqref{crudeinbnd3}. For the integration over $v$ the last factor in \eqref{crudeinbnd3} can be bounded independently of $t$ and is integrable on $[y,\infty)$. Next for the integration over $t$, if $t > MN^{-\frac{m}{2}}$ for some large constant $M$ \eqref{crudegin} implies that $w(N^{\frac{m}{2}}t)e^{\frac{m}{2}Nt^{\frac{2}{m}}} \leq C$. Likewise if $t \leq MN^{-\frac{m}{2}}$ the factor $e^{cNt^{\frac{2}{m}}}$ is uniformly bounded and the weight $w(N^{\frac{m}{2}}t)$ is integrable.   Hence the contribution to \eqref{suff} on $v > v^{*}$ is $O(e^{-cN^{\frac{2\epsilon}{m}}})$. On the other hand if $v < v^{*}$ so that $tvN^{m}$ is bounded we use that $f_{\infty}(tvN^{m}) < C$ uniformly by absolute convergence of the series $f_{\infty}$. Then the required exponential decay follows from \eqref{weightint}. A similar approach applies to the contribution coming from the second term of \eqref{Srep1}.
\end{proof}
\begin{remark}
\label{rem:saddle}
A key point in the representations \eqref{Srep1} or \eqref{prekernel} is that the integrand is dominated by points in a small vicinity of the saddle point $v=x$. A distinct advantage of representation \eqref{Srep1} over \eqref{prekernel} is that this saddle point is located outside the domain of integration, at least if $y > x$. This is a crucial point when proving exponential decay of the correlation kernel and in proving Theorem \ref{th:globintro}. In contrast, representation \eqref{prekernel} integrates fully over the saddle point making it more challenging to isolate the exponentially small contributions to the asymptotics. The representation \eqref{irep3} has a similar advantage over \eqref{ikernintro}.
\end{remark}

\begin{proof}[Proof of Theorem \ref{th:globintro}]	
We start with the proof of \eqref{Dnxyintro} based on representation \eqref{drep2}. Putting together Propositions \ref{prop:gin1} and \ref{prop:gin3}, we have the following asymptotics as $N \to \infty$,
\begin{equation}
w(N^{\frac{m}{2}}x)w(N^{\frac{m}{2}}y)f_{\infty}(N^{m}xy) = c_{m}^{(1)}N^{-\frac{3(m-1)}{2}}(xy)^{-\frac{3(m-1)}{2m}}e^{-\frac{Nm}{2}(y^{\frac{1}{m}}-x^{\frac{1}{m}})^{2}}\,(1+o(1)) 
\end{equation}
where
\begin{equation}
c_{m}^{(1)} = 2^{m-1}(2\pi)^{\frac{m-1}{2}}m^{-3/2}.
\end{equation}
Then from Proposition \ref{prop:gin2} and the crude bound \eqref{crudegin} we can estimate the error
\begin{equation}
\bigg{|}w(N^{\frac{m}{2}}x)w(N^{\frac{m}{2}}y)\left(f_{N-2}(N^{m}xy)-f_{\infty}(N^{m}xy)\right)\bigg{|} \leq N^{c}e^{-N(\phi_{m}(x)+\phi_{m}(y))} \label{fnfinfbnd}
\end{equation}
where $\phi_{m}(x) = \frac{m}{2}(x^{\frac{2}{m}}-1)-\log(x)$. Since $x < 1-N^{-\frac{1}{2}+\epsilon}$ we have $N\phi_{m}(x) \geq cN^{\epsilon}$ and so \eqref{fnfinfbnd} is $O(e^{-cN^{\epsilon}})$. Therefore,
\begin{equation}
\begin{split}
&w(N^{\frac{m}{2}}x)w(N^{\frac{m}{2}}y)f_{N-2}(N^{m}xy)\\
&= c_{m}^{(1)}N^{-\frac{3(m-1)}{2}}(xy)^{-\frac{3(m-1)}{2m}}e^{-\frac{Nm}{2}(y^{\frac{1}{m}}-x^{\frac{1}{m}})^{2}}\,(1+o(1)) + O(e^{-cN^{\epsilon}}) \label{wrwrf}
\end{split}
\end{equation}

The estimate \eqref{wrwrf} holds uniformly on $(x,y) \in E_{N}^{2}$. Replacing $x \to x^{m}$, $y \to y^{m}$ and inserting \eqref{wrwrf} into \eqref{drep2} shows that
\begin{align}
(xy)^{m-1}D_{N}(x^{m},y^{m}) &= \frac{N^{\frac{3}{2}}}{m^{\frac{3}{2}}\sqrt{2\pi}}\frac{y^{m}-x^{m}}{(xy)^{\frac{m-1}{2}}}\,e^{-\frac{Nm}{2}(x-y)^{2}}(1+o(1)) + O(e^{-N^{\epsilon}}) \label{l1dnxy}\\
&= \frac{N^{\frac{3}{2}}}{m^{\frac{3}{2}}\sqrt{2\pi}}m(y-x)\,e^{-\frac{Nm}{2}(x-y)^{2}}(1+o(1)) + O(e^{-N^{\epsilon}}), \label{l2dnxy}
\end{align}
where in the second line of \eqref{l2dnxy} we used the Taylor expansion \eqref{qsmallu} assuming $|y-x| \leq N^{-\frac{1}{2}+\epsilon'}$ and $x, y > N^{-\frac{1}{2}+\epsilon}$ with $0 < \epsilon' < \epsilon$. If on the other hand we have $|y-x| > N^{-\frac{1}{2}+\epsilon'}$ then the first line \eqref{l1dnxy} is $O(e^{-N^{\epsilon'}})$. This establishes the first formula \eqref{Dnxyintro} of the global approximation. 

To establish \eqref{Snxyintro} requires a little more work. Firstly we note that the second term in \eqref{Srep1} can be neglected in the bulk. Indeed, by Stirling's formula and the crude bound \eqref{crudegin} we obtain
\begin{equation}
C_{N,m}D_{N,m}x^{N-1}w(N^{\frac{m}{2}}x) \leq N^{c}e^{-N\phi_{m}(x)} \label{bndsecondterm}
\end{equation}
where $\phi_{m}(x) = \frac{m}{2}\left(x^{\frac{2}{m}}-1\right)-\log(x)$. Hence \eqref{bndsecondterm} is $O(e^{-cN^{\epsilon}})$ uniformly for $x \in E_{N}$.

Next we deal with the main term in \eqref{Srep1}. We begin by assuming $0<x<y$. As mentioned in Remark \ref{rem:saddle}, the main contribution to the integral in \eqref{Srep1} comes from a small neighbourhood of the point $v=x$. Hence the points that are too far away from this neighbourhood will have an exponentially small contribution. In the bounds that follow it will be helpful to define $y^{*}_{\pm} =\left(y^{\frac{1}{m}}\pm N^{-\frac{1}{2}+\epsilon'}\right)^{m}$ with $x^{*}_{\pm}$ defined similarly. In particular, inserting \eqref{wrwrf} into \eqref{Srep1}, we claim that the part of the integral where $v > y_{+}^{*}$ where $0 < \epsilon' < \epsilon$ can be neglected. Then the interval $[y,y_{+}^{*}]$ is entirely contained in the bulk region $E_{N}$ for any $\epsilon''>0$ with $0 < \epsilon'' < \epsilon' < \epsilon$. This allows us to apply the asymptotics \eqref{wrwrf}, yielding
\begin{equation}
S_{N}(x,y) = K_{N,m}\int_{y}^{y^{*}_{+}}dv\,\frac{(v-x)}{(xv)^{\frac{3(m-1)}{2m}}}\,e^{-\frac{Nm}{2}(v^{\frac{1}{m}}-x^{\frac{1}{m}})^{2}}(1+o(1)) + O(e^{-N^{\epsilon}})\label{sasympt}
\end{equation}
uniformly on $(x,y) \in E_{N}^{2}$, where 
\begin{equation}
K_{N,m} = \frac{N^{\frac{3}{2}}}{m^{\frac{3}{2}}\sqrt{2\pi}}.
%2C_{N,m}N^{-\frac{3(m-1)}{2}}c_{m}^{(1)} = 
\end{equation}
Before proving the claim \eqref{sasympt} we show how to use it to obtain the desired estimate \eqref{Snxyintro} of Theorem \ref{th:globintro}. Working in variables $x \to x^{m}$, $y \to y^{m}$ and $v \to v^{m}$, we make the further change of variables $v \to x+\frac{v}{\sqrt{N}}$ and write \eqref{sasympt} in terms of the function $Q(x,v/\sqrt{N})$ defined in Lemma \ref{lem:qbound},
\begin{align}
x^{m-1}S_{N}(x^{m},y^{m}) &= \frac{mK_{N,m}}{N}\int_{(y-x)\sqrt{N}}^{(y-x)\sqrt{N}+N^{\epsilon}}dv\,\sqrt{N}Q(x,v/\sqrt{N})\,e^{-\frac{mv^{2}}{2}}\,(1+o(1))+O(e^{-N^{\epsilon}})\\
&= \frac{mK_{N,m}}{N}\int_{(y-x)\sqrt{N}}^{(y-x)\sqrt{N}+N^{\epsilon}}dv\,mve^{-\frac{mv^{2}}{2}}\,(1+o(1))+O(e^{-N^{\epsilon}}) \label{qbnd}\\
&= \sqrt{\frac{N}{2\pi m}}e^{-Nm\frac{(x-y)^{2}}{2}}(1+o(1))+O(e^{-N^{\epsilon}}) \label{SNest}
\end{align}
where in the second line \eqref{qbnd} we used the Taylor expansion \eqref{qsmallu}. Now we prove the claim \eqref{sasympt}. Consider $A := [y_{+}^{*},\infty)$ and denote the contribution to \eqref{Srep1} with the integral over $v$ restricted to the set $A$ as $I_{A}$. On this interval we apply \eqref{crudegin} and the following uniform bound for $xv>0$, 
\begin{equation}
f_{N-2}(xvN^{m}) \leq f_{\infty}(xvN^{m}) \leq Ce^{Nm(xv)^{\frac{1}{m}}}. \label{fintbnd}
\end{equation}
Then noting that $y^{*}_{+} > x^{*}_{+}$ we get
\begin{equation}
\begin{split}
I_{A} &\leq N^{c}\int_{x^{*}_{+}}^{\infty}dv\,(v-x)e^{-\frac{Nm}{2}(v^{\frac{1}{m}}-x^{\frac{1}{m}})^{2}}\\
&\leq C(N_0)\,N^{c}e^{-\frac{(N-N_{0})m}{2}((x^{*}_{+})^{\frac{1}{m}}-x^{\frac{1}{m}})^{2}} = O(e^{-N^{\epsilon'}})
\end{split}
\end{equation}
This establishes \eqref{Snxyintro} when $0 < x < y$. 

If instead $0 < y < x$ we work with representation \eqref{Srep3}. As before, the integral is dominated by points in a small neighbourhood of $v=x$. We begin by cutting out parts of the integral far away from this neighbourhood, this time the interval $A = [0,y^{*}_{-}]$. Since $y < x$, we have 
\begin{align}
&\bigg{|}N^{c}\int_{MN^{-\frac{m}{2}-\epsilon}}^{y^{*}_{-}}dv\,(x-v)w(N^{\frac{m}{2}}x)w(N^{\frac{m}{2}}v)f_{N-2}(N^{m}xv)\bigg{|}\\
&\leq N^{c}\int_{MN^{-\frac{m}{2}-\epsilon}}^{x^{*}_{-}}dv\,e^{-\frac{Nm}{2}(x^{\frac{1}{m}}-v^{\frac{1}{m}})^{2}}w(N^{\frac{m}{2}}v)e^{\frac{Nm}{2}v^{\frac{2}{m}}}\\
&\leq N^{c}e^{-\frac{Nm}{2}(x^{\frac{1}{m}}-(x^{*}_{-})^{\frac{1}{m}})^{2}}\int_{MN^{-\frac{m}{2}-\epsilon}}^{x^{*}_{-}}dv\,w(N^{\frac{m}{2}}v)e^{\frac{Nm}{2}v^{\frac{2}{m}}}\\
&\leq CN^{c}e^{-\frac{Nm}{2}(x^{\frac{1}{m}}-(x^{*}_{-})^{\frac{1}{m}})^{2}} = O(e^{-N^{\epsilon'}}),
\end{align}
where we used \eqref{crudegin} and \eqref{fintbnd}. On the region closer to the origin we obtain
\begin{align}
&\bigg{|}N^{c}\int_{0}^{MN^{-\frac{m}{2}-\epsilon}}dv\,(x-v)w(N^{\frac{m}{2}}x)w(N^{\frac{m}{2}}v)f_{N-2}(N^{m}xv)\bigg{|}\\
&\leq C\,N^{c}e^{-\frac{Nm}{2}x^{\frac{2}{m}}}\,\int_{0}^{MN^{-\frac{m}{2}-\epsilon}}dv\,w(N^{\frac{m}{2}}v)\\
&\leq C\,N^{c}e^{-\frac{Nm}{2}x^{\frac{2}{m}}} = O(e^{-cN^{\epsilon}}),
\end{align}
where we again used \eqref{crudegin} combined with $f_{N-2}(N^{m}xv) \leq f_{\infty}(N^{m}xv) \leq C$, which holds as $N^{m}xv$ is bounded on the domain of integration. The final bound then holds because $x > N^{-\frac{m}{2}+\epsilon}$ on $E_{N}$. Likewise using \eqref{crudegin} shows that the second term in \eqref{Srep3} is exponentially small. We can thus work with the first term in \eqref{Srep3} with the integration over $v$ restricted to $v \in [y_{-}^{*},y]$ and apply the same steps that led to \eqref{SNest}. This establishes \eqref{Snxyintro} for $0 < y < x$ and completes the proof of the global approximation for $S_{N}(x,y)$.

For the $I_{N}(x,y)$ kernel it suffices to assume that $0 < x < y$ by the anti-symmetry property of $I_{N}(x,y)$, see Lemma \ref{lem:dandi}. We make use of the refined asymptotics \eqref{irep3}. For a renewed value of $\epsilon>0$ chosen so that $0 < \epsilon' < \epsilon$, we insert \eqref{SNest} into \eqref{irep3} and changing variable $t \to t^{m}$ we obtain,
\begin{align}
I_{N}(x^{m},y^{m}) &= m\int_{N^{-\frac{1}{2}+\epsilon'}}^{x}dt\,t^{m-1}S_{N}(t^{m},y^{m})+O(e^{-N^{\epsilon}})\\
&= m^{2}\frac{K_{N,m}}{N}\int_{N^{-\frac{1}{2}+\epsilon'}}^{x}dt\,e^{-\frac{N}{2}m(t-y)^{2}}(1+o(1)) + O(e^{-N^{\epsilon}})\\
&= \frac{1}{2}\,\mathrm{erfc}\left(\frac{\sqrt{Nm}(y-x)}{\sqrt{2}}\right)(1+o(1)) + O(e^{-N^{\epsilon}}).
\end{align}
This proves the asymptotics \eqref{Inxyintro}. 

So far we have assumed that $x,y > 0$. The case where $x,y < 0$ can be accessed via the symmetries $S_{N}(-x,-y) = S_{N}(x,y)$, $I_{N}(-x,-y) = -I_{N}(x,y)$ and $D_{N}(-x,-y) = -D_{N}(x,y)$. The only remaining case to deal with is when $x$ and $y$ have mixed signs. In this case we claim that as $N \to \infty$,
\begin{equation}
S_{N}(x,y) = O(e^{-N^{\epsilon}}), \qquad D_{N}(x,y) = O(e^{-N^{\epsilon}}), \qquad I_{N}(x,y) = O(e^{-N^{\epsilon}}),\end{equation}
uniformly on $|x|, |y| \in E_{N}$. First assume $x<0$ and $y>0$. To proceed we use the crude bounds \eqref{crudegin} and the bound, for any $xv<-MN^{-\frac{m}{2}}$ with $M>0$ large, there is a small constant $c>0$ such that
\begin{equation}
|f_{\infty}(N^{m}xv)| \leq e^{N(m-c)(-xv)^{\frac{1}{m}}}.\end{equation}
Combining these bounds shows that
\begin{equation}
\begin{split}
w(N^{\frac{m}{2}}x)w(N^{\frac{m}{2}}y)f_{\infty}(N^{m}xy) &\leq e^{-Nm((-x)^{\frac{1}{m}}-y^{\frac{1}{m}})^{2}-cN(-xy)^{\frac{1}{m}}}\\
&\leq e^{-\delta N ((-x)^{\frac{2}{m}}+y^{\frac{2}{m}})}
\end{split}
\end{equation}
This immediately gives the bound for $D_{N}(x,y)$, since $|x|,|y| > N^{-\frac{m}{2}+\epsilon}$. Now for $S_{N}(x,y)$, using representation \eqref{Srep1} we obtain  
\begin{equation}
\begin{split}
|S(x,y)| &\leq N^{c}e^{-\delta N ((-x)^{\frac{2}{m}}}\int_{y}^{\infty}dv\,(x-v)e^{-\delta Nv^{\frac{2}{m}}}\\
&\leq N^{c}e^{-\delta N ((-x)^{\frac{2}{m}}-\delta N y^{\frac{2}{m}}} \label{Saltsign}
\end{split}
\end{equation}
Finally, for $I_{N}(x,y)$ we use \eqref{irep2} for $x<0$, namely
\begin{equation}
I_{N}(x,y) = \int_{x}^{-MN^{-\frac{m}{2}}}dt\,S_{N}(t,y) + O(e^{-N^{\epsilon}})
\end{equation}
Now since $t<0$ and $y>0$ we can use \eqref{Saltsign} and obtain $I_{N}(x,y) = O(e^{-N^{\epsilon}})$. This completes the proof of Theorem \ref{th:globintro}.
\end{proof}

\begin{proof}[Proof of Theorem \ref{th:bulkconvintro}]
This follows from Theorem \ref{th:globintro} and the limit
\begin{equation}
\lim_{N \to \infty}\frac{Nm}{2}\left(\left(E+\frac{\xi}{2\sqrt{Nm}\rho(E)}\right)^{\frac{1}{m}}-\left(E+\frac{\zeta}{2\sqrt{Nm}\rho(E)}\right)^{\frac{1}{m}}\right)^{2} = \frac{1}{2}(\xi-\zeta)^{2},
\end{equation}
holding uniformly on compact subsets of $\xi$ and $\zeta$ for any fixed $E \in (-1,1)\setminus\{0\}$. 
\end{proof}

%Finally, we estimate the error term coming from the approximation of $f_{N-2}$ by $f_{\infty}$ on the interval $[y,y+\omega_{N}'']$. We denote this error term $J$, defined as follows
%\begin{equation}
%\begin{split}
%J &= C_{N,m}\int_{y}^{y+\omega_{N}''}dv\,(v-x)e^{-\frac{Nm}{2}(x^{\frac{2}{m}}+v^{\frac{2}{m}})}\frac{(xv)^{N-1}e^{mN}}{(2\pi N)^{m/2}(xv-1)}\\
%\leq &N^{c}e^{mN}\int_{y}^{y+\omega_{N}''}dv\,e^{-\frac{Nm}{2}(x^{\frac{2}{m}}+v^{\frac{2}{m}})}(xv)^{N-1}\\
%\leq &N^{c}e^{-cN\phi(y)}e^{-cN\phi(x)}
%\end{split}
%\end{equation}
%where $\phi(x) =\frac{m}{2}\,x^{\frac{2}{m}}-\log(x)-\frac{m}{2} \geq c(x-1)^{2}$, so that 
%\begin{equation}
%J \leq e^{-cN\omega'^{2}}
%\end{equation}
%for some constant $c>0$ independent of $x$.

%\begin{corollary}[Bulk universality]
%For a fixed $E \in (-1,1)\setminus\{0\}$, and fixed $m \in \mathbb{N}$, consider the scaled kernel \eqref{scaledK}. We have the following limit uniformly in compact subsets of $\xi$ and $\zeta$,
%\begin{equation}
%\lim_{N \to \infty}k_{N}(\xi,\zeta) = \begin{pmatrix} \frac{1}{\sqrt{2\pi}}(\zeta-\xi)e^{-\frac{1}{2}(\xi-\zeta)^{2}} & \frac{1}{\sqrt{2\pi}}\,e^{-\frac{1}{2}(\xi-\zeta)^{2}}\\ -\frac{1}{\sqrt{2\pi}}\,e^{-\frac{1}{2}(\xi-\zeta)^{2}} & \frac{1}{2}\,\mathrm{sgn}(\xi-\zeta)\,\mathrm{erfc}\left(\frac{|\xi-\zeta|}{\sqrt{2}}\right)\end{pmatrix}. \label{snconv}
%\end{equation} 
%\end{corollary}
%

\section{Central limit theorems for linear statistics of real eigenvalues}
The purpose of this section is to prove asymptotic normality of linear statistics of real eigenvalues as defined in \eqref{linstat} and to prove Theorems \ref{th:ginconv} and \ref{th:ginconvmeso}. There are three steps: a truncation, a variance calculation, and finally an estimate on the higher cumulants. Recall that our estimates on the correlation kernel are restricted to the set $\mathcal{E}_{N}$, \begin{equation}
\mathcal{E}_{N} := \{x \in [-1,1] : \{|x| > N^{-\frac{m}{2}+\epsilon}\} \wedge \{|x| < 1-N^{-\frac{1}{2}+\epsilon}\}\}. \label{calen}
\end{equation}
This is precisely the region where the global approximation Theorem \ref{th:globintro} is applicable. Relatively few real eigenvalues of the matrix $G^{(m)}$ are located in the complementary region $\mathcal{E}_{N}^{\mathrm{c}}$. Therefore we decompose the linear statistic \eqref{linstat} as $\xi_{N,m}(f) = \tilde{\xi}_{N,m}(f) +  \tilde{\xi}^{\mathrm{c}}_{N,m}(f)$, where
\begin{equation}
\tilde{\xi}_{N,m}(f) = \sum_{j=1}^{n}f(\lambda_{j})\mathbbm{1}_{\lambda_{j} \in \mathcal{E}_{N}}, \qquad \tilde{\xi}^{\mathrm{c}}_{N,m}(f) = \sum_{j=1}^{n}f(\lambda_{j})\mathbbm{1}_{\lambda_{j} \in \mathcal{E}^{\mathrm{c}}_{N}}. \label{xidef}
\end{equation}
The $L^{1}$-norm of $\xi_{N,m}(f)$ is known to be of order $\sqrt{N}$, a consequence of the known law of large numbers results for real Ginibre matrices and their products \cite{EKS94,S17,LMS21}. The $L^{1}$-norm of the complementary statistic $\tilde{\xi}^{\mathrm{c}}_{N,m}(f)$ turns out to be smaller, of order $N^{\epsilon}$. More precisely, if we have the boundedness property
\begin{equation}
\tau(f) := \sup_{x \in \mathbb{R}}\bigg\{\lvert f(x) \rvert e^{-cx^{\frac{2}{m}}}\bigg\} < \infty,
\end{equation}
it can be shown that
\begin{equation}
\mathbb{E}\left(|\tilde{\xi}^{\mathrm{c}}_{N,m}(f)|\right) < C\,\tau(f)\,N^{\epsilon}. \label{L1-statement}
\end{equation}
By choosing $\epsilon>0$ sufficiently small, the normalisations of Theorem \ref{th:ginconv} and \ref{th:ginconvmeso} are such that this term will not contribute to the distributional convergence. We give the proof of \eqref{L1-statement} in Appendix \ref{se:truncation}. 

Regarding the truncated term $\tilde{\xi}_{N,m}(f)$ of \eqref{xidef}, in Section \ref{se:variance} we calculate the variance of $\tilde{\xi}_{N,m}(f)$ and show that it grows on the order $\sqrt{N}$. Then we show that all higher cumulants of $\tilde{\xi}_{N,m}(f)$ are also $O(\sqrt{N})$, this is obtained in the final Sections 
\ref{se:cumulants}, \ref{se:clustering} and \ref{se:cumulantbound}. Such estimates show that the $k^{\mathrm{th}}$ cumulant of the normalised random variable $N^{-\frac{1}{4}}\tilde{\xi}_{N,m}(f)$ is of order $N^{\frac{1}{2}-\frac{k}{4}}$, which tends to zero for any $k\geq 3$. This implies convergence in distribution to the standard normal
\begin{equation}
\frac{\tilde{\xi}_{N,m}(f)-\mathbb{E}(\tilde{\xi}_{N,m}(f))}{\sqrt{\mathrm{Var}(\tilde{\xi}_{N,m}(f))}} \overset{d}{\longrightarrow} \mathcal{N}(0,1).
\end{equation}
By adapting this approach to the mesoscopic linear statistic $\xi^{(\tau)}_{N,m}$ we will prove Theorem \ref{th:ginconvmeso} by similar means.%Without any obvious independence between eigenvalues $\{\lambda_j\}_{j=1}^{n}$, the proof of such central limit theorems can be quite a technical task in random matrix theory. Here we will show that a simpler approach applies, made possible due to the relatively weak correlations amongst real eigenvalues of non-Hermitian random matrices. In such a circumstance there is a robust and general method that can be applied based on estimating so-called cluster functions. This approach has previously been applied to a variety of interacting particle systems exhibiting weak dependence \cite{NS12, BYY19}. However, to our knowledge it has not yet been applied to real random matrices or Pfaffian point processes. 

\subsection{Variance formulas}
\label{se:variance}
In this section we calculate the variance of the truncated linear statistic $\tilde{\xi}_{N,m}(f)$ defined in \eqref{xidef} and its mesoscopic version $\tilde{\xi}^{(\tau)}_{N,m}(f)$. Our goal is to prove the following.
\begin{proposition}
\label{prop:var}
Suppose that $f$ satisfies the assumptions of Theorem \ref{th:ginconv}. Then
\begin{equation}
\lim_{N \to \infty}N^{-\frac{1}{2}}\,\mathrm{Var}(\tilde{\xi}_{N,m}(f)) = \sqrt{\frac{2m}{\pi}}\,(2-\sqrt{2})\int_{-1}^{1}dx\,f(x)^{2}\rho(x). \label{varreals}
\end{equation}
For the mesoscopic linear statistic, suppose $f$ satisfies the assumptions of Theorem \ref{th:ginconvmeso} with fixed $E \in (-1,1)\setminus\{0\}$ and exponent $0 < \tau < \frac{1}{2}$. Then\begin{equation}
\lim_{N \to \infty}N^{-\frac{1}{2}+\tau}\,\mathrm{Var}(\tilde{\xi}^{(\tau)}_{N,m}(f)) = \sqrt{\frac{2m}{\pi}}\,(2-\sqrt{2})\,\rho(E)\int_{-\infty}^{\infty}dx\,f(x)^{2}. \label{mesovarform}
\end{equation}
On the other hand, if $E=0$ and $0 < \tau < \frac{m}{2}$ we have
\begin{equation}
\label{meso0varform}
\lim_{N \to \infty}N^{-\frac{1}{2}+\frac{\tau}{m}}\,\mathrm{Var}(\tilde{\xi}^{(\tau)}_{N,m}(f)) = \frac{1}{\sqrt{2m\pi}}\,(2-\sqrt{2})\,\int_{-\infty}^{\infty}dx\,f(x)^{2}|x|^{-1+\frac{1}{m}}.
\end{equation}
\end{proposition}

We start from a general formula for the variance of linear statistics of any Pfaffian point process with matrix kernel of the derived form:
\begin{equation}
\begin{split}
\mathrm{Var}\left(\sum_{j}f(\lambda_j)\mathbbm{1}_{\lambda_{j} \in \mathcal{E}_{N}}\right) =& \int_{\mathcal{E}_{N}}dx\,f(x)^{2}S_{N}(x,x)-\int_{\mathcal{E}_{N}^{2}}dx\,dy\,f(x)f(y)D_{N}(x,y)I_{N}(x,y)\\
&- \int_{\mathcal{E}_{N}^{2}}dx\,dy\,f(x)f(y)S_{N}(x,y)S_{N}(y,x). \label{varlinstat}
\end{split}
\end{equation}
In each of these integrals, the idea is to substitute the asymptotics of Theorem \ref{th:globintro}. From this we can immediately see that contributions to the integrals in \eqref{varlinstat} coming from opposite signs are exponentially small and can be neglected. We concentrate on the parts of the integrals where the coordinates $x$ and $y$ are both positive, as if they are both negative we use the symmetries mentioned in Theorem \ref{th:globintro}. Therefore we focus on the positive subset of $\mathcal{E}_{N}$ that we denote by $E_{N}$, as in \eqref{ENsetintro}. Throughout this Section we assume that $f$ satisfies the regularity assumed in Theorem \ref{th:ginconv}, or Theorem \ref{th:ginconvmeso} in the mesoscopic case.
For the first term in \eqref{varlinstat}, inserting the asymptotics \eqref{Snxyintro} with $x=y$, we get 
\begin{align}
\int_{E_N}dx\,f(x)^{2}S_{N}(x,x) &= m\int_{E_{N}^{\frac{1}{m}}}dx\, f(x^{m})^{2}x^{m-1}S_{N}(x^{m},x^{m}) \\
&\sim \sqrt{\frac{2Nm}{\pi}}\int_{0}^{1}dx\,f(x)^{2}\rho(x) \label{vardiag}
\end{align}
where $\rho(x) = \frac{1}{2m}\,|x|^{\frac{1}{m}-1}$ is the limiting density of real eigenvalues, and similarly for the contribution from the negative part of $\mathcal{E}_{N}$.
\begin{lemma}
\label{lem:snsnvar}
We have
\begin{equation}
\int_{\mathcal{E}_{N}^{2}}dx\,dy\,f(x)f(y)S_{N}(x,y)S_{N}(y,x) \sim \sqrt{\frac{Nm}{\pi}}\int_{-1}^{1}dx\,f(x)^{2}\rho(x) \label{snsnlim}
\end{equation}
\end{lemma}
\begin{proof}
By Theorem \ref{th:globintro} we have
\begin{align}
&m^{2}\int_{E_{N}^{\frac{1}{m}}\times E_{N}^{\frac{1}{m}}}dx\,dy\,f(x^{m})f(y^{m})(xy)^{m-1}S(x^{m},y^{m})S(y^{m},x^{m})\\
&\sim \frac{mN}{2\pi}\int_{E_{N}^{\frac{1}{m}}\times E_{N}^{\frac{1}{m}}}dx\,dy\,f(x^{m})f(y^{m})\,e^{-Nm(x-y)^{2}}. \label{RHSconv}
\end{align}
The proof of \eqref{snsnlim} now follows easily from \eqref{RHSconv} by approximating the integral near the saddle point $y=x$, but in general this would need some continuity of the function $f$. Since we have $f \in L^{2}[-1,1]$ we can do a bit better using Fourier transforms. Namely, since the function $f_{m}(x) = f(x^{m})\mathbbm{1}_{x^{m}\in E_{N}}$ belongs to $L^{2}(\mathbb{R})$ we can identify the integral over $y$ as a convolution of $f_{m}$ with the Gaussian $e^{-Nmx^{2}}$ and apply Parseval's identity. This trick was used in the $m=1$ case in \cite{K15}. We define the Fourier transform of a function $f \in L^{2}(\mathbb{R})$ by the formula, 
\begin{equation}
\hat{f}(k) := \int_{-\infty}^{\infty}dx\,f(x)e^{-2\pi i k x}.
\end{equation}
Then we can rewrite \eqref{RHSconv} as
\begin{align}
&\frac{1}{2}\sqrt{\frac{mN}{\pi}}\int_{\mathbb{R}}dk\,|\hat{f_{m}}(k)|^{2}e^{-\frac{\pi^{2}k^{2}}{Nm}}\sim \frac{1}{2}\sqrt{\frac{mN}{\pi}}\int_{\mathbb{R}}dk\,|\hat{f_{m}}(k)|^{2}\\
&= \frac{1}{2}\sqrt{\frac{mN}{\pi}}\int_{E_{N}}dx\,f(x^{m})^{2} \sim \sqrt{\frac{mN}{\pi}}\int_{0}^{1}dx\,f(x)^{2}\rho(x).
\end{align} 
This completes the proof of Lemma \ref{lem:snsnvar}.
\end{proof}

\begin{lemma}
\label{lem:dninvar}
As $N \to \infty$ we have
\begin{equation}
\begin{split}
\int_{\mathcal{E}_{N}^{2}}dxdy\,f(x)f(y)D_{N}(x,y)I_{N}(x,y)\sim -\sqrt{\frac{Nm}{2\pi}}\,(2-\sqrt{2})\,\int_{-1}^{1}dx\,f(x)^{2}\rho(x).
\end{split}
\end{equation}
\end{lemma}

\begin{proof}
Inserting the asymptotics of Theorem \ref{th:globintro} and again applying Parseval's identity to the convolution we get
\begin{align}
&m^{2}\int_{E_{N}^{\frac{1}{m}}\times E_{N}^{\frac{1}{m}}}dx\,dy\,f(x^{m})f(y^{m})(xy)^{m-1}D_{N}(x^{m},y^{m})I_{N}(x^{m},y^{m})\\
&\sim -\frac{(Nm)^{\frac{3}{2}}}{2\sqrt{2\pi}}\int_{E_{N}^{\frac{1}{m}}\times E_{N}^{\frac{1}{m}}}dx\,dy\,f(x^{m})f(y^{m})\,|x-y|\,e^{-\frac{Nm}{2}(x-y)^{2}}
\mathrm{erfc}\left(\frac{\sqrt{Nm}|x-y|
}{\sqrt{2}}\right)\\
&= -\frac{1}{2}\sqrt{\frac{Nm}{2\pi}}\int_{\mathbb{R}}dk\,|\hat{f_{m}}(k)|^{2}\,\int_{\mathbb{R}}d\xi\,e^{-\frac{2\pi i k \xi}{\sqrt{Nm}}}\,|\xi|e^{-\frac{\xi^{2}}{2}}\,\mathrm{erfc}\left(\frac{|\xi|}{\sqrt{2}}\right) \\
&\sim -\frac{1}{2}\sqrt{\frac{Nm}{2\pi}}\int_{\mathbb{R}}dk\,|\hat{f_{m}}(k)|^{2}\,\int_{\mathbb{R}}d\xi\,|\xi|e^{-\frac{\xi^{2}}{2}}\,\mathrm{erfc}\left(\frac{|\xi|}{\sqrt{2}}\right) \label{parsednin}\\
&\sim -\sqrt{\frac{Nm}{2\pi}}\,(2-\sqrt{2})\,\int_{0}^{1}dx\,f(x)^{2}\rho(x), \label{finaldnin}
\end{align}
where we used the exact formula
\begin{equation}
\int_{-\infty}^{\infty}dx\,|\xi| \,e^{-\frac{1}{2}\xi^{2}}
\mathrm{erfc}\left(\frac{|\xi|
}{\sqrt{2}}\right) = 2-\sqrt{2},
\end{equation}
which is a simple consequence of integration by parts.
\end{proof}
\begin{proof}[Proof of Proposition \ref{prop:var}]
The first variance formula \eqref{varreals} follows immediately from putting together the preceeding Lemmas \ref{lem:dninvar}, \ref{lem:snsnvar}, \eqref{vardiag} and inserting them into the variance formula \eqref{varlinstat}. For the mesoscopic variance formulas \eqref{mesovarform} and \eqref{meso0varform} the same estimates show that for any fixed $E \in (-1,1)$ and $\tau>0$ we have
\begin{equation}
\mathrm{Var}(\tilde{\xi}_{N,m}^{(\tau)}) \sim \sqrt{\frac{2Nm}{\pi}}\,(2-\sqrt{2})\,\int_{-1}^{1}dx\,f\left(N^{\tau}(E-x)\right)^{2}\rho(x). \label{varmesoproof}
\end{equation} 
If $E \neq 0$ we use the assumptions of Theorem \ref{th:ginconvmeso}, namely the bound $\sup_{x \in \mathbb{R}}((1+|x|)^{1+\delta}f(x)^{2}) < \infty$ to control the part of the integral where $|x-E| > \delta_{0}$ for some $\delta_{0}$ chosen sufficiently small that the interval $[E-\delta_{0},E+\delta_{0}]$ does not intersect the origin. Then we have
\begin{align}
&\int_{-1}^{1}dx\,f\left(N^{\tau}(E-x)\right)^{2}\rho(x)\sim \int_{E-\delta_{0}}^{E+\delta_{0}}dx\,f\left(N^{\tau}(E-x)\right)^{2}\rho(x)\\
&= N^{-\tau}\int_{-N^{\tau}\delta_{0}}^{N^{\tau}\delta_{0}}dx\,f(x)^{2}\rho(E+xN^{-\tau}) \sim N^{-\tau}\,\rho(E)\int_{-\infty}^{\infty}dx\,f(x)^{2},
\end{align}
which follows from dominated convergence and the assumptions on $f$ in Theorem \ref{th:ginconvmeso}. This gives \eqref{mesovarform}. The proof of \eqref{meso0varform} follows similarly from \eqref{varmesoproof} and the estimate
\begin{equation}
\int_{-1}^{1}dx\,f(N^{\tau}x)^{2}\rho(x) \sim N^{-\frac{\tau}{m}}\,\frac{1}{2m}\,\int_{-\infty}^{\infty}dx\,f(x)^{2}\,|x|^{-1+\frac{1}{m}}.
\end{equation}
 \end{proof}

\subsection{Cumulants and cluster functions}
\label{se:cumulants}
Higher moments and cumulants of the linear statistic $\tilde{\xi}_{N,m}(f)$ depend on knowledge of $k$-point correlations functions $\rho^{(k)}_{N}(x_1,\ldots,x_k)$. In particular, the cumulants are most naturally expressed in terms of a variant of the $\rho^{(k)}_{N}$ known as \textit{cluster functions}, denoted $r^{(k)}_{N}(x_1,\ldots,x_k)$. The relation between cluster functions and correlation functions is analogous to the relation between ordinary cumulants and moments of a random variable. Although they are not used very frequently in random matrix theory, their properties were discussed in some detail already by Mehta \cite{Meh04}. They are sometimes also referred to as \textit{Ursell functions} or \textit{truncated $k$-point functions} \cite{NS12,BYY19}. To define them, we first introduce the relevant notation. 

Following \cite{NS12}, let $\Pi(k,j)$ denote the collection of all unordered partitions of the set $\{1,2,\ldots,k\}$ into $j$ non-empty disjoint blocks, and by $\Pi(k)$ the collection of all unordered partitions of the set $\{1,2,\ldots,k\}$ into non-empty disjoint blocks. For $\pi \in \Pi(k,j)$, we denote the blocks by $\{\pi_{1},\ldots,\pi_j\}$ with an arbitrarily chosen enumeration, and denote the lengths of the blocks by $p_t = |\pi_t|, 1 \leq t \leq j$. The number of blocks in the partition $\pi$ will be
denoted by $|\pi|$.

\begin{definition}
\label{def:cluster}
The $k$-point cluster function $r^{(k)}_{N}(x_1,\ldots,x_k)$ is a symmetric function of the configuration $X := (x_1,\ldots,x_k)$ defined by the formula
\begin{equation}
r^{(k)}_{N}(x_1,\ldots,x_k) = \sum_{\ell=1}^{k}(-1)^{\ell-1}(\ell-1)!\sum_{\pi \in \Pi(k,\ell)}\rho^{(p_1)}_{N}(X_{p_1})\ldots \rho^{(p_\ell)}_{N}(X_{p_\ell})
\end{equation}
where $X_{p_j} = (x_{i} : i \in \pi_j)$.
\end{definition}
As we are working with a Pfaffian point process, the correlation functions $\rho^{(k)}_{N}$ and cluster functions $r^{(k)}_{N}$ are completely described by the corresponding kernel $K_{N}(x,y)$. Therefore, we will often speak about the cluster functions \textit{corresponding to a given kernel}.

In general, having a good control on $k$-point cluster functions translates into having good control on the $k^{\mathrm{th}}$ cumulant, due to the following result.
\begin{lemma}
The $k^{\mathrm{th}}$ cumulant of the linear statistic $\tilde{\xi}_{N,m}(f)$, denoted $C_{k}(f)$, can be expressed in terms of cluster functions via
\begin{equation}
C_{k}(f) = \sum_{\pi \in \Pi(k)}\int_{\mathcal{E}_{N}^{n}}\prod_{j=1}^{n}dx_{j}\,f(x_j)^{p_j}\,r^{(n)}_{N}(x_1,\ldots,x_n) \label{cexpr}
\end{equation}
where $n$ is the number of blocks in the partition $\pi$.
\end{lemma}
\begin{proof}
This is a general relation between cumulants of a linear statistic and cluster functions. See for example \cite[Claim 4.3]{NS12}.
\end{proof}

Considering the pre-factors in Theorem \ref{th:globintro}, it will be useful in what follows to normalise the kernels by the limiting density, $\rho(x) = \frac{1}{2m}|x|^{-1+\frac{1}{m}}$. We also pre-multiply $D_{N}$ and $I_{N}$ by $N^{-\frac{1}{2}}$ and $N^{\frac{1}{2}}$ respectively. This motivates us to introduce a modified kernel 
\begin{equation}
\tilde{K}_{N}(x,y) = N^{-\frac{1}{2}}\begin{pmatrix} N^{-\frac{1}{2}}\rho(x)^{-1}\rho(y)^{-1}D_{N}(x,y) & \rho(x)^{-1}S_{N}(x,y)\\ -\rho(y)^{-1}S_{N}(y,x) & N^{\frac{1}{2}}I_{N}(x,y)\end{pmatrix}. \label{tildekn}
\end{equation}
Furthermore, by removing common factors of $\rho(x_j)$ and $N^{\frac{1}{2}}$ from alternate rows and columns of the Pfaffian, we have the identity
\begin{equation}
\mathrm{Pf}\{K_{N}(x_i,x_j)\}_{i,j=1}^{p} = \left(\prod_{j=1}^{p}\sqrt{N}\rho(x_j)\right)\,\mathrm{Pf}\{\tilde{K}_{N}(x_i,x_j)\}_{i,j=1}^{p}.
\end{equation}
We denote by $\tilde{\rho}^{(k)}_{N}$ and $\tilde{r}^{(k)}_{N}$ the correlation and cluster functions corresponding to the kernel \eqref{tildekn}. Then the cumulant formula \eqref{cexpr} becomes
\begin{equation}
C_{k}(f) = \sum_{\pi \in \Pi(k)}N^{\frac{n}{2}}\int_{\mathcal{E}_{N}^{n}}\prod_{j=1}^{n}dx_{j}\,\rho(x_j)f(x_j)^{p_j}\,\tilde{r}^{(n)}_{N}(x_1,\ldots,x_n). \label{cexpr2}
\end{equation}
The advantage of the new kernel $\tilde{K}_{N}(x,y)$ is that it normalises the singular powers of $x$ and $y$ that arise in the asymptotics of Theorem \ref{th:globintro}. This comes at the expense of the product of singular terms $\prod_{j=1}^{n}\rho(x_j)$ appearing in the integration measure. As we shall see later, when expressed in variables $x_{j} \to x_{j}^{m}$ the Jacobian of this change of variables perfectly cancels such singular terms.

\subsection{Clustering property of the eigenvalues}
\label{se:clustering}
The main input for bounding the cumulants in \eqref{cexpr2} will be to obtain a useful bound on the cluster functions. First we recall what it means for a point process to be clustering. Let $\phi, \tau : \mathbb{R} \to \mathbb{R}_+$ be fast decreasing functions, such as those that decay faster than any polynomial.
\begin{definition}
\label{def:clustering}
We say that the $k$-point correlation functions of a point process depending on a parameter $N$ are \textit{asymptotically clustering} if for each partition of the set of indices $\{1,\ldots,k\}$ into non-empty disjoint subsets $I$ and $J$, one has for sufficiently large $N$
\begin{equation}
|\rho_{N}(\xi_1,\ldots,\xi_{k})-\rho_{N}(\xi_{I})\rho_{N}(\xi_{J})| \leq \phi(\mathrm{Dist}(\xi_{I},\xi_{J})) + \tau(N)
\end{equation}
where $\xi_{I} = (\xi_{i} : i \in I)$ and $\mathrm{Dist}(\xi_I, \xi_J) = \min_{i \in I, j \in J} \lvert \xi_i - \xi_j \rvert$.
\end{definition}
To obtain estimates of this type in the context of determinantal point processes, the following Lemma is useful, see for example \cite[Section 5.2]{BYY19}.
\begin{lemma}
\label{AGZ_lemma}
Let $K$ and $L$ be $2 \times 2$ matrix kernels. For a positive integer $p$ and any $(x_1, \ldots, x_p)$ we have,
denoting $\norm{K} = \sup_{1 \leq i,j \leq p}\sup_{l, m \in \{1, 2\}}  \lvert [K(x_i,x_j)]_{l, m} \rvert$
\begin{equation}
\bigg{|}\mathrm{det}\{K(x_i, x_j)\}_{i,j=1}^{p} - \mathrm{det}\{L(x_i, x_j)\}_{i,j=1}^{p}\bigg{|} \leq (2p)^{1+p} \norm{K-L} \max(\norm{K},  \norm{L})^{2p-1}.
\end{equation}
\end{lemma}
\begin{proof}
See \cite[Lemma 3.3]{AGZ10}.
\end{proof}
We now show that the above Lemma leads to a simple criterion for a Pfaffian point process to have asymptotically clustering $k$-point functions.
\begin{lemma}
\label{le:pfaffcriteria}
Given a Pfaffian point process with $2 \times 2$ matrix kernel $\tilde{k}_{N}(\xi,\zeta)$ defined on a configuration space $\mathcal{E}_{N}$, suppose that for any $i,j \in \{1,2\}$ we have the bound
\begin{equation}
|[\tilde{k}_{N}(\xi,\zeta)]_{i,j}| \leq \phi(|\xi-\zeta|) + \tau(N)
\end{equation}
uniformly in $\xi, \zeta \in \mathcal{E}_{N}$. Then the $k$-point correlation functions corresponding to the kernel $\tilde{k}_{N}$ are asymptotically clustering with fast decreasing functions given by constant multiples of $\phi^{\frac{1}{2}}$ and $\tau^{\frac{1}{2}}$.
\end{lemma}

\begin{proof}
Without loss of generality we consider $k=p+q$ and consider the product of $p$-point and $q$-point correlation functions,
\begin{align*}
\rho_N^{(p)}(\xi_1, \ldots, \xi_{p}) \rho_N^{(q)}(\xi_{p+1}, \ldots, \xi_{p+q})  & = 
%\text{pf}(\mathbf{K}^{(n)}(x_i, x_j): 1 \leq i, j \leq p) \text{pf}(\mathbf{K}^{(n)}(x_i, x_j): p+1 \leq i, j \leq p+q)\\
\mathrm{pf}(A) \mathrm{pf}(B) \\
 & = \text{pf}\left(\begin{matrix} A & 0 \\
0 & B 
\end{matrix}\right) \\
& = \text{det}^{1/2}\left(\begin{matrix} A & 0 \\
0 & B 
\end{matrix}\right)
\end{align*}
where $A$ is the $2p \times 2p$ matrix $(\tilde{k}_{N}(\xi_i, \xi_j):1\leq i, j \leq p)$ and $B$ is the 
$2q \times 2q$ matrix $(\tilde{k}_{N}(\xi_i, \xi_j):p+1\leq i, j \leq p+q)$. 

For all $(\xi_1, \ldots, \xi_{p+q}) \in \mathcal{E}_n^{p+q}$,
\begin{align*}
& \lvert \rho_N^{(p+q)}(\xi_1, \ldots, \xi_{p+q}) - \rho_N^{(p)}(\xi_1, \ldots, \xi_{p}) \rho_N^{(q)}(\xi_{p+1}, \ldots, \xi_{p+q})  \rvert \\
& =\bigg{\lvert} \text{det}^{1/2}(\tilde{k}_{N}(\xi_i, \xi_j):1\leq i, j \leq p+q) - \text{det}^{1/2}\left(\begin{matrix} A & 0 \\
0 & B 
\end{matrix}\right)
 \bigg{\rvert} \\
& \leq \bigg \lvert \text{det}(\tilde{k}_{N}(\xi_i, \xi_j):1\leq i, j \leq p+q) - \text{det}\left(\begin{matrix} A & 0 \\
0 & B 
\end{matrix}\right)
 \bigg\rvert^{1/2} \\
 & \leq C_{k}
 %\sup_{(i,j) \in \{p+1,\ldots,k\} \times \{1,\ldots,p\}}
 \sup_{(i,j) \in \{1,\ldots,p\} \times \{p+1,\ldots,k\} \cup \{p+1,\ldots,k\} \times \{1,\ldots,p\} }
\sup_{l,m \in \{1,2\}}|(\tilde{k}_{N}(\xi_i, \xi_j))_{l,m}|^{\frac{1}{2}} \\
& \leq C_{k}\phi(\mathrm{Dist}(\{\xi_1,\ldots,\xi_p\},\{\xi_{p+1},\ldots,\xi_k\}))^{\frac{1}{2}} + C_{k}\tau(N)^{\frac{1}{2}}
\end{align*}
by using Lemma \ref{AGZ_lemma} with constant $C_{k}=(2k)^{\frac{k+1}{2}}\sup_{l,m \in \{1,2\}}\sup_{i,j \in \{1,\ldots,k\}}|\tilde{k}_{N}(\xi_i, \xi_j)_{l,m}
|^{k - \frac{1}{2}} < \infty$. This establishes the required decay with fast decreasing functions $C_k \phi^{1/2}$ and $C_k \tau^{1/2}$.
\end{proof}

\begin{corollary}
\label{corx}
Consider a Pfaffian point process with asymptotically clustering and bounded $k$-point correlation functions with the fast decreasing functions from Definition \ref{def:clustering} given by constant multiples of $\phi^{1/2}$ and $\tau^{1/2}$.  
%consider the corresponding cluster functions $r^{(k)}_{N}(\xi_1,\ldots,\xi_k)$. 
Then there exist constants $c_{k}, \kappa_k >0$ such that for any configuration $(\xi_1,\ldots,\xi_k)$ the corresponding cluster functions satisfy
\begin{equation}
|r^{(k)}_{N}(\xi_1,\ldots,\xi_{k})| \leq \kappa_k \phi^{\frac{1}{2}}(c_{k}\,\mathrm{diam}(\xi_1,\ldots,\xi_k))+ \kappa_k \tau^{\frac{1}{2}}(N)
\end{equation}
where $\mathrm{diam}(\xi_1,\ldots,\xi_k) = \mathrm{max}_{1 \leq i,j \leq k}|\xi_j-\xi_i|$.
\end{corollary}
\begin{proof}
The proof is identical to the one given in \cite[Proof of Claim 4.1]{NS12}. For ease of reading we summarise the key points. The relationship between correlation and cluster functions means that for any 
partition of $\{1, \ldots, k\}$ into non-empty subsets $I$ and $J$ we have 
%$J \subset \{1, \ldots, k\}$ 
%and $J^c = \{1, \ldots, k\} \setminus J$ with $J, J^c$ non-empty we have
\begin{equation}
\label{ind_ursell}
r_N^{(k)}(\xi_1, \ldots, \xi_k) = \rho_N^{(k)}(\xi_1, \ldots, \xi_k) - \rho_N^{(\lvert I \rvert)}(\xi_i : i \in I) \rho_N^{(\lvert J \rvert)}(\xi_i : i \in J) + \sum_{\pi \in \Pi^*} \prod_{i=1}^{n} r_N^{(p_i)}(\xi_{p_i})
\end{equation}
where $\pi$ is a partition with $n$ blocks $\{\pi_1, \ldots, \pi_n\}$ of lengths $p_i = \lvert \pi_i \rvert$ for $i =1, \ldots, n$ and $\xi_{p_j} = (\xi_i : i \in \pi_j).$ The sum is taken over the restricted set of partitions that mix $I$ and $J$, meaning $\Pi^*$ is the set of partitions of $\{1, \ldots, k\}$ such that at least one block in $\Pi^*$ contains an element from both $I$ and $J$. 
Then using asymptotic clustering of correlation functions 
\begin{align*}
 \rho_N^{(k)}(\xi_1, \ldots, \xi_k) - \rho_N^{(\lvert I \rvert)}(\xi_i : i \in I) \rho_N^{(\lvert J \rvert)}(\xi_i : i \in J)
 & \leq  C_k \phi^{\frac{1}{2}}(\mathrm{Dist}((\xi_i : i \in I), (\xi_i : i \in J))+ C_k \tau^{\frac{1}{2}}(N). 
\end{align*}
We proceed inductively as the final term in \eqref{ind_ursell} involves partitions which mix $I$ and $J$, thus the same argument applies to the cluster functions corresponding to a block with an element from both $I$ and $J$. All the remaining cluster functions are bounded. Therefore an inductive argument establishes that
\begin{equation}
|r^{(k)}_{N}(\xi_1,\ldots,\xi_{k})| \leq \kappa_k \phi^{\frac{1}{2}}(\mathrm{Dist}((\xi_i : i \in I), (\xi_i : i \in J))+ \kappa_k \tau^{\frac{1}{2}}(N).
\end{equation}
The proof is completed since there exists $c_k > 0$ such that for any $(\xi_1, \ldots, \xi_k)$ we can choose $I$ and $J$ such that the inequality $\mathrm{Dist}((\xi_i : i \in I), (\xi_i : i \in J)) \geq c_k \mathrm{diam}(\xi_1,\ldots,\xi_k)$ holds. 
\end{proof}

\subsection{Bound on the higher cumulants and proof of Theorem \ref{th:ginconv}}
\label{se:cumulantbound}
The purpose of this section is to prove the following.
\begin{theorem}
\label{th:cumulantbnd}
For positive integers $n$, $p_1,\ldots,p_n$ and a function $f \in L^{\infty}([-1,1])$, there is a constant $C>0$ such that 
\begin{equation}
N^{\frac{n}{2}}\bigg{|}\int_{\mathcal{E}_{N}^{n}}\prod_{j=1}^{n}dx_{j}\,\rho(x_j)f(x_j)^{p_j}\,\tilde{r}^{(n)}_{N}(x_1,\ldots,x_n)\bigg{|} \leq C\sqrt{N} \label{clusterint}
\end{equation}
and consequently the $k^{\mathrm{th}}$ cumulant \eqref{cexpr2} is of order $C_{k}(f) = O(\sqrt{N})$ as $N \to \infty$.
\end{theorem}
\begin{proof}
Writing $\mathcal{E}_{N} = (-E_{N})\cup E_{N}$, we decompose the integral in \eqref{clusterint} into $2^{n}$ domains comprising of positive and negative parts of $\mathcal{E}_{N}$. Without loss of generality we suppose that the first $q$ coordinates are positive and the last $n-q$ are negative, with $q=0,\ldots,n$. Then it suffices to estimate the integral
\begin{equation}
I_{N,q}(f) = N^{\frac{n}{2}}\int_{E_{N}^{q} \times (-E_{N})^{n-q}}\prod_{j=1}^{n}dx_{j}\,\rho(x_j)f(x_j)^{p_j}\,\tilde{r}^{(n)}_{N}(x_1,\ldots,x_n).
\end{equation}
To make contact with Theorem \ref{th:globintro} we make the change of variable
\begin{equation}
x_{j} = s_{\xi_j}\left(\frac{|\xi_{j}|}{\sqrt{N}}\right)^{m},\qquad j=1,\ldots,m \label{chvar}
\end{equation}
where $s_{\xi} = \mathrm{sgn}(\xi)$. We define a corresponding new domain $F_{N} = \sqrt{N}E_{N}^{\frac{1}{m}}$ and
 $\tilde{F}_{N} = F_{N}\cup(-F_{N})$. We also define 
$h(\xi) = f\left(s_{\xi}\left(\frac{\lvert \xi \rvert}{\sqrt{N}}\right)^{m}\right)$. 
%$h(\xi) = f\left(s_{\xi}\left(s_{\xi}\frac{\xi}{\sqrt{N}}\right)^{m}\right)$. 
Finally, let $R^{(n)}_{N}$ be the cluster function defined by the new coordinates \eqref{chvar}, i.e. replacing all variables $x_{j}$ in $\tilde{r}^{(n)}_{N}$ with $\xi_{j}$ according to \eqref{chvar}. 

 The Jacobian of transformation \eqref{chvar} satisfies $2dx_{j}\,\rho(x_j) = \frac{d\xi_{j}}{\sqrt{N}}$ and we arrive at
\begin{align}
|I_{N,q}(f)| &\leq C\int_{F_{N}^{q}\times (-F_{N})^{n-q}}\prod_{j=1}^{n}d\xi_{j}\,|h^{p_j}(\xi_{j})|\,|R^{(n)}_{N}(\xi_1,\ldots,\xi_n)|\\
&\leq C\norm{f}_{\infty}^{n}\sqrt{N}\sup_{\xi_{n} \in \tilde{F}_{N}}\int_{\tilde{F}_{N}^{n-1}}d\xi_{1}\ldots d\xi_{n-1}\,|R^{(n)}_{N}(\xi_1,\ldots,\xi_n)|. \label{2ndbnd}
\end{align}
In the new variables \eqref{chvar} the correlation kernel $\tilde{K}_{N}(x_1,x_2)$ becomes
\begin{equation}
\tilde{k}_{N}(\xi_1,\xi_2) = \tilde{K}_{N}\left(s_{\xi_1}\left(\frac{|\xi_{1}|}{\sqrt{N}}\right)^{m},s_{\xi_2}\left(\frac{|\xi_{2}|}{\sqrt{N}}\right)^{m}\right).
\end{equation}
Then from the definition \eqref{tildekn} of $\tilde{K}_{N}$ and Theorem \ref{th:globintro} we see that the transformed kernel $\tilde{k}_{N}(\xi_1,\xi_2)$ satisfies the following exponential decay, uniformly on $\tilde{F}_{N} \times \tilde{F}_{N}$:
\begin{equation}
\sup_{i,j \in \{1,2\}}\,|[\tilde{k}_{N}(\xi_1,\xi_2)]_{i,j}| \leq Ce^{-c|\xi_1-\xi_2|^{2}} + O(e^{-N^{\epsilon}}). \label{expdecayktilde}
\end{equation}
By Lemma \ref{le:pfaffcriteria} this implies that the $k$-point correlation functions corresponding to the kernel $\tilde{k}_{N}(\xi,\zeta)$ are asymptotically clustering. Note that \eqref{expdecayktilde} also implies that the $k$-point correlation functions are uniformly bounded. Then Corollary \ref{corx} gives the bound
\begin{equation}
|R^{(n)}_{N}(\xi_1,\ldots,\xi_n)| \leq Ce^{-c\mathrm{diam}(\vec{\xi})^{2}} + O(e^{-N^{\epsilon}})\label{diambnd}.
\end{equation}
To finish the proof, we decompose $\mathbb{R}^{n-1} = \bigcup_{l=1}^{\infty}G_{l}$ where
\begin{equation}
G_{l} = \bigg\{(\xi_1,\ldots,\xi_{n-1}) \in \mathbb{R}^{n-1} : l-1 \leq \mathrm{diam}(\xi_1,\ldots,\xi_{n}) \leq l\bigg\}.
\end{equation}
Note that $|G_{l}| \leq (2l)^{n-1}$ and $|R_{N}(\xi_1,\ldots,\xi_{n})| \leq Ce^{-c(l-1)^{2}}$ for $(\xi_1,\ldots,\xi_{n-1}) \in G_{l}$. Then the supremum in \eqref{2ndbnd} is bounded by
\begin{equation}
C\sum_{l=1}^{\infty}(2l)^{n-1}e^{-c(l-1)^{2}} + O(|F_{N}|^{n}e^{-N^{\epsilon}}) < \infty.
\end{equation}
This completes the proof of Theorem \ref{th:cumulantbnd}.
\end{proof}
We are now ready to prove Theorems \ref{th:ginconv} and \ref{th:ginconvmeso}.
\begin{proof}[Proof of Theorems \ref{th:ginconv} and \ref{th:ginconvmeso}]
Normalising \eqref{xidef} by the appropriate power of $N$, we have
\begin{equation}
\frac{\xi_{N,m}(f)}{N^{\frac{1}{4}}} = \frac{\tilde{\xi}_{N,m}(f)}{N^{\frac{1}{4}}} + \frac{\tilde{\xi}^{\mathrm{c}}_{N,m}(f)}{N^{\frac{1}{4}}}. \label{xnf}
\end{equation}
By Theorem \ref{th:cumulantbnd} and the limiting formula \eqref{varreals} for the variance, we have that Theorem \ref{th:ginconv} holds for the truncated statistic $\tilde{\xi}_{N,m}(f)$. By Lemma \ref{le:l1est} the second term in \eqref{xnf} tends to zero in $L^{1}$ and hence can be neglected for the purposes of distributional convergence (this is sometimes referred to as Slutsky's theorem). This completes the proof of Theorem \ref{th:ginconv}. The proof of Theorem \ref{th:ginconvmeso} is analogous; we write
\begin{equation}
\frac{\xi^{(\tau)}_{N,m}(f)}{N^{\frac{1-2\tau}{4}}} = \frac{\tilde{\xi}^{(\tau)}_{N,m}(f)}{N^{\frac{1-2\tau}{4}}} + \frac{\tilde{\xi}^{(\tau),\mathrm{c}}_{N,m}(f)}{N^{\frac{1-2\tau}{4}}}. \label{xnfmeso}
\end{equation}
Then the second term of \eqref{xnfmeso} again tends to zero in $L^{1}$, where this time we have to choose $\epsilon>0$ sufficiently small for a given $\tau$ in Lemma \ref{le:l1est}, at least so that $0 < \epsilon < \frac{1-2\tau}{4}$. For the first term of \eqref{xnfmeso} we have the limiting variance given in \eqref{mesovarform}. For the higher cumulants, since $f$ is uniformly bounded the same estimates of Theorem \ref{th:cumulantbnd} apply as for the $\tau=0$ case. In other words we have that the $k^{\mathrm{th}}$ cumulant of $N^{-\frac{1-2\tau}{4}}\tilde{\xi}^{(\tau)}_{N,m}(f)$ is of order $O\left(N^{\frac{1}{2}-k\left(\frac{1-2\tau}{4}\right)}\right)$ as $N\to \infty$. This implies that for all $0 < \tau < \frac{1}{2}$, there exists a $K$ such that for all $k \geq K$, the $k^{\mathrm{th}}$ cumulant of $N^{\frac{1-2\tau}{4}}\tilde{\xi}^{(\tau)}_{N,m}(f)$ tends to zero as $N \to \infty$. Then the convergence \eqref{normalconvmeso} follows from Marcinkiewicz's theorem \cite{M39}.
\end{proof}
We also state here another central limit theorem not covered by Theorem \ref{th:ginconvmeso}. It pertains to the case $E=0$ of \eqref{mesolinstat}.
\begin{theorem}
\label{th:mesozero}
Let $E=0$ in \eqref{mesolinstat} and assume $0 < \tau < \frac{m}{2}$ and that the regularity assumptions of Theorem \ref{th:ginconvmeso} hold. Then we have the convergence in distribution to a normal random variable,
\begin{equation}
\frac{\xi^{(\tau)}_{N,m}(f)-\mathbb{E}(\xi^{(\tau)}_{N,m}(f))}{N^{\frac{1}{4}-\frac{\tau}{2m}}} \overset{d}{\longrightarrow} \mathcal{N}(0,\sigma_{0}^{2}(f)), \qquad N \to \infty, \label{normalconvmesoorigin}
\end{equation}
with limiting variance
\begin{equation}
\sigma_{0}^{2}(f) = \frac{1}{\sqrt{2m\pi}}\,(2-\sqrt{2})\,\int_{-\infty}^{\infty}dx\,f(x)^{2}|x|^{-1+\frac{1}{m}}.
\end{equation}
\end{theorem}
\begin{proof}
The proof follows exactly as for Theorem \ref{th:ginconvmeso}, except one uses the variance formula  \eqref{meso0varform}. Then the same estimates on the cumulants \eqref{clusterint} and Marcinkiewicz's theorem completes the proof.
\end{proof}

\section{Fluctuations and correlations at the edge}
\label{se:edge}
In this Section we consider the scaled kernels at the edge. In terms of the kernel $K_{N}(x,y)$ of Theorem \ref{th:ik} and its entries given by \eqref{prekernel}, \eqref{dkernintro} and \eqref{ikernintro}, we define
\begin{align}
s^{\mathrm{edge}}_{N}(\xi,\zeta) &= \frac{1}{2\sqrt{Nm}\rho(1)}\,S_{N}\left(1+\frac{\xi}{2\sqrt{Nm}\rho(1)},1+\frac{\zeta}{2\sqrt{Nm}\rho(1)}\right), \label{snedgeform}\\
d^{\mathrm{edge}}_{N}(\xi,\zeta) &= \frac{1}{4Nm\rho(1)^{2}}\,D_{N}\left(1+\frac{\xi}{2\sqrt{Nm}\rho(1)},1+\frac{\zeta}{2\sqrt{Nm}\rho(1)}\right), \label{dnedgeform}\\
i^{\mathrm{edge}}_{N}(\xi,\zeta) &= I_{N}\left(1+\frac{\xi}{2\sqrt{Nm}\rho(1)},1+\frac{\zeta}{2\sqrt{Nm}\rho(1)}\right).\label{inedgeform}
\end{align}
The density at the edge is $\rho(1) = \frac{1}{2m}\,|E|^{\frac{1}{m}-1}\bigg{|}_{E=1} = \frac{1}{2m}$ and we interpret $\sqrt{Nm}$ as the asymptotic number of particles. Let us now construct the corresponding limits. We have
\begin{align}
s_{\infty}^{\mathrm{edge}}(\xi,\zeta) &:= \frac{1}{2\sqrt{2\pi}}\,e^{-\frac{1}{2}(\xi-\zeta)^{2}}\mathrm{erfc}\left(\frac{\xi+\zeta}{\sqrt{2}}\right)+\frac{1}{4\sqrt{\pi}}\,e^{-\xi^{2}}\mathrm{erfc}(-\zeta), \label{sedgeform}\\
d_{\infty}^{\mathrm{edge}}(\xi,\zeta) &:= \frac{1}{2\sqrt{2\pi}}\,(\zeta-\xi)e^{-\frac{1}{2}(\xi-\zeta)^{2}}\mathrm{erfc}\left(\frac{\xi+\zeta}{\sqrt{2}}\right), \label{dedgeform}\\
i_{\infty}^{\mathrm{edge}}(\xi,\zeta) &:= \int_{\xi}^{\zeta}dt\,s^{\mathrm{edge}}(t,\zeta) + \frac{1}{2}\mathrm{sgn}(\xi-\zeta). \label{iedgeform}
\end{align}
The main goal of this Section will be to show that the scaled finite-$N$ kernels above in \eqref{snedgeform}, \eqref{dnedgeform} and \eqref{inedgeform} converge to their respective limits \eqref{sedgeform}, \eqref{dedgeform} and \eqref{iedgeform} in a suitably strong way. By showing that this convergence holds uniformly on intervals of unbounded height, we will be able to transfer this knowledge onto a convergence result for the largest real eigenvalue. Such a strategy was also employed for the largest eigenvalue of the GOE, see in particular \cite[Theorem 3.9.24]{AGZ10} and surrounding discussion.

One difficulty to overcome can be observed at the level of the limiting kernels above, particularly in \eqref{sedgeform}. From these expressions it is not immediately clear that the correlation functions decay rapidly for large values of their argument. For example when $\xi$ is fixed and $\zeta \to \infty$ the kernel \eqref{sedgeform} does not tend to zero. However, if one considers \textit{e.g.} the $2$-point correlation function, then due to pairings of type $s_{\infty}^{\mathrm{edge}}(\xi,\zeta)s_{\infty}^{\mathrm{edge}}(\zeta,\xi)$, one has exponential decay if either $\xi \to \infty$ or $\zeta \to \infty$.

In order to clarify this exponential decay at the level of kernels, it will be helpful to conjugate them by the following function:
\begin{equation}
\nu_{N}(\xi) := e^{-\frac{\sqrt{Nm}}{2}\left(\left(1+\sqrt{\frac{m}{N}}\xi\right)^{\frac{2}{m}}-1\right)}. \label{nudef}
\end{equation}
For large $N$ this function is a close approximation of the function $e^{-\xi}$, at least on scales $\xi < N^{\epsilon}$. More precisely, a Taylor expansion shows that
\begin{equation}
\sup_{\xi \in (s,N^{\epsilon})}|\nu_{N}(\xi)-e^{-\xi}| = O(e^{-\xi}N^{-\frac{1}{2}+\epsilon}), \qquad N \to \infty.
\end{equation}
This function also has good $L^{1}$-integrability properties. The same Taylor expansion and a saddle point approximation shows that
\begin{equation}
\norm{\nu_{N}}_{1} := \int_{s}^{\infty}d\xi\,\nu_{N}(\xi) = e^{-s}+O(N^{-\frac{1}{2}}). \label{finitenu}
\end{equation}
In the following it will be helpful to define the supremum norm,
\begin{equation}
\norm{K}_{(s,\infty)} := \sup_{(\xi,\zeta) \in (s,\infty)^{2}}\max_{i,j\in\{1,2\}}|K(\xi,\zeta)_{i,j}|.
\end{equation}
We are now ready to define the conjugated kernels:
\begin{equation}
\begin{split}
\tilde{d}_{N}^{\mathrm{edge}}(\xi,\zeta) &= \frac{1}{\nu_{N}(\xi)\nu_{N}(\zeta)}\,d^{\mathrm{edge}}_{N}(\xi,\zeta)\\
\tilde{s}_{N}^{\mathrm{edge}}(\xi,\zeta) &= \frac{1}{\nu_{N}(\xi)}\,s^{\mathrm{edge}}_{N}(\xi,\zeta)\\
\tilde{i}^{\mathrm{edge}}_{N}(\xi,\zeta) &= i_{N}^{\mathrm{edge}}(\xi,\zeta)
\end{split}
\end{equation}
and similarly for the limiting kernels we conjugate in the same way by $\nu_{N}$. The corresponding matrix kernels will be denoted $\tilde{K}^{\mathrm{edge}}_{N}(\xi,\zeta)$ and $\tilde{K}_{\infty}^{\mathrm{edge}}(\xi,\zeta)$ respectively (note that both kernels depend on $N$). We have the following relation between Pfaffians of the unconjugated and conjugated kernels
\begin{equation}
\mathrm{Pf}\bigg\{K^{\mathrm{edge}}_{N}(\xi_{j},\xi_{k})\bigg\}_{j,k=1}^{\ell} = \prod_{j=1}^{N}\nu_{N}(\xi_j)\,\mathrm{Pf}\bigg\{\tilde{K}^{\mathrm{edge}}_{N}(\xi_{j},\xi_{k})\bigg\}_{j,k=1}^{\ell}.
\label{nupfaffident}
\end{equation}
%Then we define the matrix kernel
%\begin{equation}
%K^{\mathrm{edge}}_{N}(\xi,\zeta)  = \begin{pmatrix} d^{\mathrm{edge}}_{N}(\xi,\zeta) & s^{\mathrm{edge}}_{N}(\xi,\zeta)\\ -s^{\mathrm{edge}}_{N}(\zeta,\xi) & i^{\mathrm{edge}}_{N}(\xi,\zeta)\end{pmatrix}
%\end{equation}
Then Theorem \ref{th:edgeuniformityintro} is an immediate consequence of the following stronger result:
\begin{theorem}
\label{th:edgeuniformity}
For a fixed $s \in \mathbb{R}$ consider the set $A = \{ \xi \in \mathbb{R} : \xi > s\}$. Then for a fixed $m \in \mathbb{N}$, we have $\tilde{s}_{N}^{\mathrm{edge}}, \tilde{d}_{N}^{\mathrm{edge}}$ and $\tilde{i}_{N}^{\mathrm{edge}}$ converging to their limits $\tilde{s}_{\infty}^{\mathrm{edge}}, \tilde{d}_{\infty}^{\mathrm{edge}}$ and $\tilde{i}_{\infty}^{\mathrm{edge}}$ uniformly on $(\xi,\zeta) \in A^{2}$ as $N \to \infty$. In other words we have,
\begin{align}
&\lim_{N \to \infty}\sup_{(\xi,\zeta) \in A^{2}}|\tilde{s}_{N}^{\mathrm{edge}}(\xi,\zeta) - \tilde{s}_{\infty}^{\mathrm{edge}}(\xi,\zeta)| = 0, \label{tildesconv}\\
&\lim_{N \to \infty}\sup_{(\xi,\zeta) \in A^{2}}|\tilde{d}_{N}^{\mathrm{edge}}(\xi,\zeta) - \tilde{d}_{\infty}^{\mathrm{edge}}(\xi,\zeta)| = 0, \label{tildedconv}\\
&\lim_{N \to \infty}\sup_{(\xi,\zeta) \in A^{2}}|\tilde{i}_{N}^{\mathrm{edge}}(\xi,\zeta) - \tilde{i}_{\infty}^{\mathrm{edge}}(\xi,\zeta)| = 0. \label{tildeiconv}
\end{align}

\end{theorem}

Before giving the proof we show how to use it to obtain convergence in distribution of the largest real eigenvalue and prove Theorem \ref{th:lambdamax}.
\begin{proof}[Proof of Theorem \ref{th:lambdamax}]
It is known from the general theory of point processes that the probability of a gap can be expressed as a series involving correlation functions (\textit{e.g.} Chapter 5 of \cite{DVJ88}). We have,
\begin{align}
& \mathbb{P}\left(\sqrt{\frac{N}{m}}\,\left(\lambda^{(m)}_{N,\mathrm{max}}-1\right) < s\right) = \mathbb{P}\left(\bigg\{\textrm{the interval}\,\, [1+\sqrt{\frac{m}{N}}\,s,\infty) \,\,\textrm{contains no eigenvalues}\bigg\}\right)\\
 &=\sum_{\ell=1}^{N}\frac{(-1)^{\ell}}{\ell!}\int_{1+\sqrt{\frac{m}{N}}\,s}^{\infty}\ldots \int_{1+\sqrt{\frac{m}{N}}\,s}^{\infty}\,\mathrm{Pf}\bigg\{K_{N}(x_{j},x_{k})\bigg\}_{j,k=1}^{\ell}\,dx_{1}\ldots dx_{\ell}\\
&=\sum_{\ell=1}^{N}\frac{(-1)^{\ell}}{\ell!}\int_{s}^{\infty}\ldots \int_{s}^{\infty}\,\left(\prod_{j=1}^{\ell}\nu_{N}(\xi_j)\right)\,\mathrm{Pf}\bigg\{\tilde{K}^{\mathrm{edge}}_{N}(\xi_{j},\xi_{k})\bigg\}_{j,k=1}^{\ell}\,d\xi_{1}\ldots d\xi_{\ell}.
\end{align}
where we use identity \eqref{nupfaffident} to obtain the final line. The corresponding limiting distribution is given again using \eqref{nupfaffident} by
\begin{align}
\mathbb{P}\left(\lambda_{\mathrm{max}} < s\right) 
&=\sum_{\ell=1}^{\infty}\frac{(-1)^{\ell}}{\ell!}\int_{s}^{\infty}\ldots \int_{s}^{\infty}\,\left(\prod_{j=1}^{\ell}\nu_{N}(\xi_j)\right)\,\mathrm{Pf}\bigg\{\tilde{K}_{\infty}^{\mathrm{edge}}(\xi_{j},\xi_{k})\bigg\}_{j,k=1}^{\ell}\,d\xi_{1}\ldots d\xi_{\ell}.\label{limitpfaff}
\end{align}

Then using the identity $\mathrm{Pf}(M)^{2} = \det(M)$ for a square anti-symmetric matrix $M$, combined with the inequality $|x-y| \leq |x^{2}-y^{2}|^{\frac{1}{2}}$ for non-negative $x$ and $y$, we get the bound
\begin{align}
&\bigg{|}\mathbb{P}\left(\sqrt{\frac{N}{m}}\,\left(\lambda^{(m)}_{N,\mathrm{max}}-1\right) < s\right)-\mathbb{P}\left(\lambda_{\mathrm{max}} < s\right)\bigg{|}\\
&\leq \sum_{\ell=1}^{N}\frac{1}{\ell!}\int_{[s,\infty)^{l}}\,\left(\prod_{j=1}^{\ell}d\xi_{j}\,\nu_{N}(\xi_j)\right)\,\bigg{|}\mathrm{Pf}\bigg\{\tilde{K}^{\mathrm{edge}}_{N}(\xi_{j},\xi_{k})\bigg\}_{j,k=1}^{\ell}-\mathrm{Pf}\bigg\{\tilde{K}_{\infty}^{\mathrm{edge}}(\xi_{j},\xi_{k})\bigg\}_{j,k=1}^{\ell}\bigg{|}+e_{N}\\
&\leq\sum_{\ell=1}^{N}\frac{1}{\ell!}\int_{[s,\infty)^{l}}\,\left(\prod_{j=1}^{\ell}d\xi_{j}\,\nu_{N}(\xi_j)\right)\,\bigg{|}\mathrm{det}\bigg\{\tilde{K}^{\mathrm{edge}}_{N}(\xi_{j},\xi_{k})\bigg\}_{j,k=1}^{\ell}-\mathrm{det}\bigg\{\tilde{K}^{\mathrm{edge}}_{\infty}(\xi_{j},\xi_{k})\bigg\}_{j,k=1}^{\ell}\bigg{|}^{\frac{1}{2}}+e_{N}\\
&\leq \norm{\tilde{K}^{\mathrm{edge}}_{N}-\tilde{K}^{\mathrm{edge}}_{\infty}}_{(s,\infty)}^{\frac{1}{2}}\,\sum_{\ell=1}^{\infty}\norm{\nu_{N}}_{1}^{\ell}\,\mathrm{max}^{\ell}\left(\norm{\tilde{K}^{\mathrm{edge}}_{N}}_{(s,\infty)},\norm{\tilde{K}_{\infty}^{\mathrm{edge}}}_{(s,\infty)}\right)\,\frac{(2\ell)^{\frac{1}{2}+\frac{\ell}{2}}}{\ell!}+e_{N}\label{kerndiffbnd}
\end{align}
where
\begin{equation}
e_{N} := \sum_{\ell=N+1}^{\infty}\frac{1}{\ell!}\int_{s}^{\infty}\ldots \int_{s}^{\infty}\,\left(\prod_{j=1}^{\ell}\nu_{N}(\xi_j)\right)\,\mathrm{Pf}\bigg\{\tilde{K}_{\infty}^{\mathrm{edge}}(\xi_{j},\xi_{k})\bigg\}_{j,k=1}^{\ell}\,d\xi_{1}\ldots d\xi_{\ell},
\end{equation}
and where the final bound \eqref{kerndiffbnd} follows from Lemma \ref{AGZ_lemma}. Note that by Theorem \ref{th:edgeuniformity} the norms $\norm{\tilde{K}^{\mathrm{edge}}_{N}}_{(s,\infty)}$ and $\norm{\tilde{K}^{\mathrm{edge}}_{\infty}}_{(s,\infty)}$ are uniformly bounded. Combined with the finiteness of $\norm{\nu_{N}}_{1}$ as in \eqref{finitenu} and Stirling's formula we deduce that the summation in \eqref{kerndiffbnd} is uniformly bounded.  The same reasoning shows that $e_{N} \to 0$ as $N \to \infty$ and that the series \eqref{limitpfaff} converges absolutely. Then Theorem \ref{th:edgeuniformity} shows that $\norm{\tilde{K}^{\mathrm{edge}}_{N}-\tilde{K}_{\infty}^{\mathrm{edge}}}_{(s,\infty)} \to 0$ as $N \to \infty$ and completes the proof.
\end{proof}

\begin{proof}[Proof of Theorem \ref{th:edgeuniformity}]
We work with the scaling $x = 1+\frac{\xi}{2\sqrt{Nm}\rho(1)}$, $y = 1+\frac{\zeta}{2\sqrt{Nm}\rho(1)}$.
We first consider the part of $A^{2}$ where $-s < \xi < N^{\epsilon}$ and $-s < \zeta < N^{\epsilon}$ where $\epsilon$ is chosen sufficiently small, \textit{e.g.} $\epsilon < \frac{1}{4}$ will be enough. Thus we define the truncation $\tilde{A} = \{-s < \xi < N^{\epsilon}\}$. The following estimates are easily established, uniformly for $\xi \in \tilde{A}$,
\begin{align}
\frac{1}{\nu_{N}(\xi)} &\sim e^{\xi}, \\
w(N^{\frac{m}{2}}x) &\sim e^{-\frac{Nm}{2}}\,N^{-\frac{m-1}{2}}\,e^{-\sqrt{Nm}\xi+\frac{1}{2}\xi^{2}(m-2)}\frac{(4\pi)^{\frac{m-1}{2}}}{\sqrt{m}} \label{weight2}
\end{align}
The truncated series $f_{N-2}(N^{m}xv)$ will be estimated using Lemma \ref{lem:fnedgelem}. Taking $v=\frac{w}{2\rho(1)\sqrt{Nm}}$ and further Taylor expanding the main result of that Lemma in the $w$, $\xi$ and $\zeta$ variables, we get
\begin{equation}
\begin{split}
f_{N-2}(N^{m}x(v+y)) &\sim \frac{e^{Nm}}{(2\pi)^{\frac{m}{2}}}\sqrt{\frac{\pi}{2m}}\,N^{-\frac{m-1}{2}}e^{\sqrt{Nm}(\xi+\zeta+w)-\frac{1}{2}m(\xi+\zeta+w)^{2}+m\xi(\zeta+w)}\\
&\times e^{\frac{1}{2}(\xi+\zeta+w)^{2}}\mathrm{erfc}\left(\frac{\xi+\zeta+w}{\sqrt{2}}\right) \label{erfcest}
\end{split}
\end{equation}
which also holds uniformly in $(\xi,\zeta,w) \in \tilde{A}^{3}$. 

Then combining \eqref{erfcest} and \eqref{weight2} we have
\begin{equation}
\begin{split}
&\frac{1}{\nu_{N}(\xi)\nu_{N}(\zeta))}w(N^{\frac{m}{2}}x)w(N^{\frac{m}{2}}(v+y))f_{N-2}(N^{m}x(v+y))\\
&\sim e^{\xi+\zeta}\sqrt{\frac{2}{\pi}}\,\frac{(2\sqrt{2\pi})^{m}}{8m^{\frac{3}{2}}}\,N^{-\frac{3(m-1)}{2}}e^{-\frac{1}{2}(\xi-\zeta-w)^{2}}\mathrm{erfc}\left(\frac{\xi+\zeta+w}{\sqrt{2}}\right) \label{wwfasy}
\end{split}
\end{equation}
For $v=w=0$ this gives $\tilde{d}^{\mathrm{edge}}_{N}(\xi,\zeta) \sim \tilde{d}^{\mathrm{edge}}(\xi,\zeta)$ uniformly on $\tilde{A}^{2}$ as $N \to \infty$. Then since $\tilde{d}^{\mathrm{edge}}(\xi,\zeta)$ is uniformly bounded we get \eqref{tildedconv} restricted to $\tilde{A}^{2}$. 

For the $s^{\mathrm{edge}}_{N}(\xi,\zeta)$ kernel we use the representation \eqref{Srep1} but we truncate the integral at level $y+\frac{N^{\epsilon}}{\sqrt{Nm}\rho(1)}$. Then the asymptotic \eqref{wwfasy} remains applicable. After the change of variables $v = y+\frac{w}{2\sqrt{Nm}\rho(1)}$ we obtain
\begin{equation}
\begin{split}
&\frac{2}{\nu_{N}(\xi)}\,\frac{N^{\frac{3m}{2}}}{(2\sqrt{2\pi})^{m}}\int_{y}^{y+\frac{N^{\epsilon}}{\sqrt{Nm}\rho(1)}}dv\,(x-v)w(N^{\frac{m}{2}}x)w(N^{\frac{m}{2}}v)f_{N-2}(N^{m}xv)\\
&\sim 2e^{\xi}\sqrt{Nm}\rho(1)\,\frac{1}{2\sqrt{2\pi}}\,\int_{0}^{N^{\epsilon}}dw\,(\xi-\zeta-w)e^{-\frac{1}{2}(\xi-\zeta-w)^{2}}\,\mathrm{erfc}\left(\frac{\xi+\zeta+w}{\sqrt{2}}\right)\\
&= 2e^{\xi}\sqrt{Nm}\rho(1)\,\frac{1}{2\sqrt{2\pi}}\,\int_{0}^{\infty}dw\,(\xi-\zeta-w)e^{-\frac{1}{2}(\xi-\zeta-w)^{2}}\,\mathrm{erfc}\left(\frac{\xi+\zeta+w}{\sqrt{2}}\right)+O(e^{-cN^{2\epsilon}}),
\label{truncateint}
\end{split}
\end{equation}
where we used the Gaussian decay of the complementary error function to neglect the tail of the integral, and that $e^{\xi-N^{2\epsilon}} = O(e^{-cN^{2\epsilon}})$. The second term in \eqref{Srep1} also contributes. Combining Stirling's formula, \eqref{cnmdef} and \eqref{weight2} establishes the following uniform convergence,
\begin{equation}
\lim_{N \to \infty}\sup_{s < \xi < N^{\epsilon}}\bigg{|}\frac{C_{N,m}x^{N-1}N^{\frac{m(N-3)}{2}}2^{m\left(\frac{N-1}{2}\right)}}{\nu_{N}(\xi)\,2\sqrt{Nm}\rho(1)}\left(\frac{\Gamma\left(\frac{N-1}{2}\right)}{(N-2)!}\right)^{m}w_{r}(N^{\frac{m}{2}}x)-\frac{1}{\sqrt{4\pi}}\,e^{\xi-\xi^{2}}\bigg{|}=0. \label{unifxiterm}
\end{equation}
The limiting kernel for $\tilde{s}_{N}^{\mathrm{edge}}(\xi,\zeta)$ is thus
\begin{equation}
\begin{split}
&\frac{e^{\xi}}{2\sqrt{2\pi}}\int_{0}^{\infty}dw\,(\xi-\zeta-w)e^{-\frac{1}{2}(\xi-\zeta-w)^{2}}\,\mathrm{erfc}\left(\frac{\xi+\zeta+w}{\sqrt{2}}\right)+\frac{1}{\sqrt{4\pi}}\,e^{\xi-\xi^{2}}\\
&=\frac{e^{\xi}}{2\sqrt{2\pi}}\left(e^{-\frac{1}{2}(\xi-\zeta)^{2}}\mathrm{erfc}\left(\frac{\xi+\zeta}{\sqrt{2}}\right)+\sqrt{\frac{2}{\pi}}\int_{0}^{\infty}dw\,e^{-\frac{1}{2}(\xi-\zeta-w)^{2}}e^{-\frac{1}{2}(\xi+\zeta+w)^{2}}\right)+ \frac{1}{\sqrt{4\pi}}\,e^{\xi-\xi^{2}}\\
&= \frac{e^{\xi}}{2\sqrt{2\pi}}\,e^{-\frac{1}{2}(\xi-\zeta)^{2}}\mathrm{erfc}\left(\frac{\xi+\zeta}{\sqrt{2}}\right)+\frac{1}{4\sqrt{\pi}}\,e^{\xi-\xi^{2}}\mathrm{erfc}(-\zeta), \label{limitkernalt}
\end{split}
\end{equation}
where we have used integration by parts. 

Now we control the complement of the integration range in \eqref{truncateint}, namely $[y+\frac{N^{\epsilon}}{\sqrt{Nm}\rho(1)},\infty)$. Define
\begin{equation}
\psi_{N}(x) = \frac{m}{2}(x^{\frac{2}{m}}-1)-\log(x)\left(1-\frac{a}{N}\right)-\frac{m}{2\sqrt{N}}(x^{\frac{2}{m}}-1),
\end{equation}
where $a \in \mathbb{R}$ is fixed. The function $\psi_{N}(x)$ has a unique minimum at the point
\begin{equation}
x^{*} = \left(\frac{1+\frac{a}{N}}{1-N^{-\frac{1}{2}}}\right)^{\frac{m}{2}}
\end{equation}
and monotonically increases beyond $x > x^{*}$. Furthermore, a simple Taylor expansion shows that $N\psi_{N}(x^{*}) = O(1)$. If $x > 1+N^{\epsilon-\frac{1}{2}}$ then $e^{-N\psi_{N}(x)} \leq e^{-N\psi_{N}(1+N^{\epsilon-\frac{1}{2}})} = O(e^{-\frac{1}{m}N^{2\epsilon}})$. 

Similarly, define
\begin{equation}
\phi(v) = \frac{m}{2}(v^{\frac{2}{m}}-1)-\log(v).
\end{equation}
This function has a unique minimum at $v=1$ and increases monotonically for $v >1$. If $y>1$, then in the definition of $\tilde{s}_{N}(x,y)$ we have $v\geq y$ and so $\phi(v) \geq \phi(y)$. Hence the middle term in the decomposition
\begin{equation}
N\phi(v) = (N-1)\phi(y) + (N-1)(\phi(v)-\phi(y)) + \phi(v) \label{saddledecomp}
\end{equation}
is non-negative and $e^{-N\phi(v)} \leq e^{-(N-1)\phi(y)}\,e^{-\phi(v)}$. 

Then we can bound
\begin{equation}
\begin{split}
&\frac{1}{\nu_{N}(\xi)}\,N^{c}\int_{y+N^{-\frac{1}{2}+\epsilon}}^{\infty}dv\,|x-v|w(N^{\frac{m}{2}}x)w(N^{\frac{m}{2}}v)f_{N-2}(N^{m}xv)\\
&\leq N^{c}\int_{y+N^{-\frac{1}{2}+\epsilon}}^{\infty}dv\,e^{-N\phi(v)-N\psi_{N}(\xi)}\\
&\leq N^{c}e^{-(N-1)\phi\left(y+N^{-\frac{1}{2}+\epsilon}\right)-N\psi_{N}(\xi)}\int_{1}^{\infty}dv\,e^{-\phi(v)}\\
&=O(e^{-cN^{2\epsilon}})
\end{split}
\end{equation}

For the integrated kernel $i^{\mathrm{edge}}_{N}(\xi,\zeta)$, we have
\begin{align}
&\sup_{(\xi,\zeta) \in \tilde{A}^{2}}\bigg{|}i^{\mathrm{edge}}_{N}(\xi,\zeta)-i^{\mathrm{edge}}(\xi,\zeta)\bigg{|}\\
&= \sup_{(\xi,\zeta) \in \tilde{A}^{2}}\bigg{|}\bigg{|}\int_{\xi}^{\eta}d\tau\,\bigg{|}s^{\mathrm{edge}}_{N}(\tau,\zeta)-s^{\mathrm{edge}}(\tau,\zeta)\bigg{|}\bigg{|}\\
&= \sup_{(\xi,\zeta) \in \tilde{A}^{2}}\bigg{|}\bigg{|}\int_{\xi}^{\eta}d\tau\,\nu_{N}(\tau)\bigg{|}\tilde{s}^{\mathrm{edge}}_{N}(\tau,\zeta)-\tilde{s}^{\mathrm{edge}}(\tau,\zeta)\bigg{|}\bigg{|}\\
&\leq \norm{\nu_N}_{1}\,\sup_{(\xi,\zeta) \in \tilde{A}^{2}}\bigg{|}\tilde{s}^{\mathrm{edge}}_{N}(\xi,\zeta)-\tilde{s}^{\mathrm{edge}}(\xi,\zeta)\bigg{|} \to 0, \qquad N \to \infty,
\end{align}
where the last limit follows from the already obtained uniform convergence of $\tilde{s}^{\mathrm{edge}}_{N}(\xi,\zeta)$ on $\tilde{A}^{2}$.

We now deal with the complement of $\tilde{A}^{2}$, namely when either $x > 1+N^{\epsilon-\frac{1}{2}}$ or $y > 1+N^{\epsilon-\frac{1}{2}}$. We have
\begin{equation}
\begin{split}
&|x-y|e^{\sqrt{N}\frac{m}{2}(x^{\frac{2}{m}}-1)+\sqrt{N}\frac{m}{2}(y^{\frac{2}{m}}-1)}w(N^{\frac{m}{2}}x)w(N^{\frac{m}{2}}y)f_{N-2}(N^{m}xy)\\
&\leq N^{c}e^{-N\psi_{N}(x)-N\psi_{N}(y)}\\
&= O(e^{-cN^{2\epsilon}}).
\end{split}
\end{equation} 
This shows that $\tilde{d}_{N}(x,y) = O(e^{-cN^{2\epsilon}})$ uniformly on $(x,y) \in \tilde{A}^{2}$. Now for the kernel $\tilde{s}_{N}(x,y)$, we have
\begin{align}
&e^{\sqrt{N}\frac{m}{2}(x^{\frac{2}{m}}-1)}|x-v|w(N^{\frac{m}{2}}v)w(N^{\frac{m}{2}}x)f_{N-2}(N^{m}vx) \leq N^{c}e^{-N\psi_{N}(x)}(x+v)e^{-N\phi(v)}\\
&\leq N^{c}e^{-N\psi_{N}(x)-(N-1)\phi(y)}(x+v)e^{-\phi(v)}.
\end{align}
and therefore if $(x,y) \in A^{2}\setminus\tilde{A}^{2}$, we have
\begin{equation}
\begin{split}
&\bigg{|}e^{\frac{m}{2}\sqrt{N}(x^{\frac{2}{m}}-1)}\int_{y}^{\infty}dv\,(x-v)w(N^{\frac{m}{2}}x)w(N^{\frac{m}{2}}(v))f_{N-2}(N^{m}xv)\bigg{|}\\
&\leq N^{c}e^{-N\psi_{N}(x)-(N-1)\phi(y)}\int_{1}^{\infty}dv\,(x+v)e^{-\phi(v)}\label{stildebnd}\\
&=O(e^{-N^{2\epsilon}}).
\end{split} 
\end{equation}
For the second term in \eqref{Srep1}, an application of Stirling's formula and the crude bound on the weight \eqref{crudegin} shows that for any $x > 1+N^{\epsilon-\frac{1}{2}}$, we have
\begin{equation}
\begin{split}
&e^{\sqrt{N}(x-1)}\bigg{|}\frac{C_{N,m}x^{N-1}N^{\frac{m(N-3)}{2}}2^{m\left(\frac{N-1}{2}\right)}}{2\sqrt{Nm}\rho(1)}\left(\frac{\Gamma\left(\frac{N-1}{2}\right)}{(N-2)!}\right)^{m}w_{r}(N^{\frac{m}{2}}x)\bigg{|} \leq N^{c}e^{-N\psi_{N}(x)}\\
&=O(e^{-N^{2\epsilon}}).
\end{split}
\end{equation}
For the $I_{N}(x,y)$ kernel suppose that $y>1+N^{\epsilon-\frac{1}{2}}$. Then inserting \eqref{Srep1} into \eqref{ikernintro} and repeating the bounds in $\eqref{stildebnd}$ shows that the contribution from the first term in \eqref{Srep1} is exponentially small. 
%\begin{equation}
%|S_{N}(t,y)| \leq N^{c}e^{-N\phi(t)-(N-1)\phi(y)} \leq N^{c}e^{-\phi(t)-(N-1)\phi(y)}
%\end{equation}
%so that
%\begin{align}
%\int_{x}^{y}dt\,|S_{N}(t,y)| \leq e^{-(N-1)\phi(y)}\int_{1-\delta}^{\infty}e^{-\phi(t)} = O(e^{-N^{2\epsilon}}).
%\end{align}
%If $x > 1+N^{-\frac{1}{2}+\epsilon}$ the same strategy shows that 
%\begin{equation}
%\int_{x}^{\infty}dt\,S_{N}(t,y) = O(e^{-N^{2\epsilon}})
%\end{equation}
%Then we have
%\begin{align}
%I_{N}(x,y) &= \int_{y}^{x}dt\,S_{N}(t,y)-\frac{1}{2} = -\int_{y}^{\infty}dt\,S_{N}(t,y)+\int_{x}^{\infty}dt\,S_{N}(t,y)-\frac{1}{2}\\
%&= C_{N,m}N^{\frac{m(N-3)}{2}}2^{m\left(\frac{N-1}{2}\right)}\left(\frac{\Gamma\left(\frac{N-1}{2}\right)}{(N-2)!}\right)^{m}\int_{y}^{\infty}dt\,t^{N-1}w_{r}(N^{\frac{m}{2}}t) -\frac{1}{2}+ O(e^{-N^{2\epsilon}})
%\end{align}
Similarly, employing the representation on the middle line of \eqref{limitkernalt} for the limiting kernel, we have
\begin{align}
i^{\mathrm{edge}}(\xi,\zeta) &= \int_{\zeta}^{\xi}d\tau\,s^{\mathrm{edge}}(\tau,\zeta)-\frac{1}{2}\\
&= \frac{1}{\sqrt{4\pi}}\int_{\xi}^{\infty}d\tau\,e^{-\tau^{2}}-\frac{1}{2}+O(e^{-cN^{2\epsilon}})
%&=\frac{1}{\sqrt{4\pi}}\int_{\zeta}^{\infty}d\tau\,e^{-\tau^{2}}-\frac{1}{2} + O(e^{-N^{2\epsilon}})
\end{align}
We therefore have
\begin{align}
&\bigg{|}I_{N}(x,y)-i^{\mathrm{edge}}(\xi,\zeta)\bigg{|}\\
&\leq \bigg{|}C_{N,m}N^{\frac{m(N-3)}{2}}2^{m\left(\frac{N-1}{2}\right)}\left(\frac{\Gamma\left(\frac{N-1}{2}\right)}{(N-2)!}\right)^{m}\int_{x}^{\infty}dt\,t^{N-1}w_{r}(N^{\frac{m}{2}}t) - \frac{1}{\sqrt{4\pi}}\int_{\xi}^{\infty}e^{-\tau^{2}}\,d\tau\bigg{|} + O(e^{-cN^{2\epsilon}})
\end{align}
If $x = 1+\sqrt{\frac{m}{N}}\,\xi$ with $\xi<N^{\epsilon}$, the latter difference tends to zero uniformly by the same approach used to derive \eqref{unifxiterm}. If $\xi > N^{\epsilon}$ both terms are $O(e^{-cN^{2\epsilon}})$.
\end{proof}

\appendix

\section{Microscopic limit at the origin}
\label{se:smallest}
In this section we prove Theorems \ref{th:origkern}  and \ref{th:smalleigintro}. In terms of the scalar kernel $S_{N}(x,y)$ of Theorem \ref{th:ik} and definitions \eqref{dkernintro} and \eqref{ikernintro} we define
\begin{align}
d^{(\mathrm{origin})}_{N,m}(\xi,\zeta) &= N^{-m}\,D_{N}(\xi N^{-\frac{m}{2}},\zeta N^{-\frac{m}{2}}),\\
s^{(\mathrm{origin})}_{N,m}(\xi,\zeta) &= N^{-\frac{m}{2}}\,S_{N}(\xi N^{-\frac{m}{2}},\zeta N^{-\frac{m}{2}}),\\
i^{(\mathrm{origin})}_{N,m}(\xi,\zeta) &= I_{N}(\xi N^{-\frac{m}{2}},\zeta N^{-\frac{m}{2}}),
\end{align}
and
\begin{equation}
k^{(\mathrm{origin})}_{N,m}(\xi,\zeta) = \begin{pmatrix} d^{(\mathrm{origin})}_{N,m}(\xi,\zeta) & s^{(\mathrm{origin})}_{N,m}(\xi,\zeta)\\ -s^{(\mathrm{origin})}_{N,m}(\zeta,\xi) & i^{(\mathrm{origin})}_{N,m}(\xi,\zeta) \end{pmatrix}
\end{equation}
Then we define a limiting kernel by the following formulas
\begin{align}
d^{(\mathrm{origin})}_{\infty,m}(\xi,\zeta) &= 2\frac{1}{(2\sqrt{2\pi})^{m}}(\zeta-\xi)w(\xi)w(\zeta)f_{\infty}(\xi \zeta),\\
s^{(\mathrm{origin})}_{\infty,m}(\xi,\zeta) &= \frac{1}{(2\sqrt{2\pi})^{m}}\,\int_{\mathbb{R}}d\eta\,(\xi-\eta)\mathrm{sgn}(\zeta-\eta)w(\xi)w(\eta)f_{\infty}(\xi \eta),\\
i^{(\mathrm{origin})}_{\infty,m}(\xi,\zeta) &= -\int_{\xi}^{\zeta}dt\,s^{(\mathrm{origin})}_{\infty,m}(t,\zeta)+\frac{1}{2}\,\mathrm{sgn}(\xi-\zeta),
\end{align}
and
\begin{equation}
k^{(\mathrm{origin})}_{\infty,m}(\xi,\zeta) = \begin{pmatrix} d^{(\mathrm{origin})}_{\infty,m}(\xi,\zeta) & s^{(\mathrm{origin})}_{\infty,m}(\xi,\zeta)\\ -s^{(\mathrm{origin})}_{\infty,m}(\zeta,\xi) & i_{\infty,m}^{(\mathrm{origin})}(\xi,\zeta) \end{pmatrix}. \label{koriglim}
\end{equation}
\begin{proposition}
\label{prop:kernolim}
The limit $\lim_{N \to \infty}k_{N}^{(\mathrm{origin})}(\xi,\zeta) = k^{(\mathrm{origin})}_{\infty,m}(\xi,\zeta)$ holds uniformly on compact subsets of $\xi$ and $\zeta$.
\end{proposition}

\begin{proof}
We begin with the kernel $D_{N}(x,y)$. By considering the difference $f_{\infty}(\xi \zeta) - f_{N-2}(\xi \zeta)$ combined with absolute convergence of $f_{\infty}$, it is straightforward to check that the following holds uniformly on compact subsets of $\xi$ and $\zeta$:
\begin{equation}
\lim_{N \to \infty}2N^{-m}C_{N,m}(y-x)w(N^{\frac{m}{2}}x)w(N^{\frac{m}{2}}y)f_{N-2}(N^{m}xy) = 2\frac{1}{(2\sqrt{2\pi})^{m}}(\zeta-\xi)w(\xi)w(\zeta)f_{\infty}(\xi \zeta)
\end{equation}
For the $S_{N}$ kernel we change variables $v = \eta N^{-\frac{m}{2}}$ and obtain
\begin{equation}
s^{(\mathrm{origin})}_{N,m}(\xi,\zeta) = \frac{1}{(2\sqrt{2\pi})^{m}}\,\int_{\mathbb{R}}d\eta\,(\xi-\eta)\mathrm{sgn}(\zeta-\eta)w(\xi)w(\eta)f_{N-2}(\xi \eta).
\end{equation}
Similarly, the difference between finite-$N$ and limiting kernel is governed by the tail sum $f_{\infty}(\xi \eta) - f_{N-2}(\xi \eta)$. Using this and the integrability of the weights we get the uniform convergence
\begin{equation}
\lim_{N \to \infty}s^{(\mathrm{origin})}_{N,m}(\xi,\zeta) = \frac{1}{(2\sqrt{2\pi})^{m}}\,\int_{\mathbb{R}}d\eta\,(\xi-\eta)\mathrm{sgn}(\zeta-\eta)w(\xi)w(\eta)f_{\infty}(\xi \eta).\label{sorigconv}
\end{equation}
Finally, the uniform convergence \eqref{sorigconv} easily implies uniform convergence of the integrated kernel $i^{(\mathrm{origin})}_{N,m}(x,y) \to i^{(\mathrm{origin})}_{\infty,m}(\xi,\zeta)$ as $N \to \infty$.
\end{proof}

Let $\lambda^{(m)}_{\mathrm{min},N}$ denote the smallest positive real eigenvalue of the product matrix $G^{(m)}$. 
\begin{proposition}
We have the convergence in distribution
\begin{equation}
N^{\frac{m}{2}}\lambda^{(m)}_{\mathrm{min},N} \overset{d}{\longrightarrow} \lambda^{(m)}_{\mathrm{min}}, \qquad N \to \infty,
\end{equation}
where the limiting random variable is defined in terms of the kernel \eqref{koriglim} by the absolutely convergent series,
\begin{equation}
\mathbb{P}(\lambda^{(m)}_{\mathrm{min}} > s) = \sum_{\ell=1}^{\infty}\frac{(-1)^{\ell}}{\ell!}\,\int_{[0,s]^{\ell}}\prod_{j=1}^{\ell}d\xi_{j}\,\mathrm{Pf}\bigg\{k^{(\mathrm{origin})}_{\infty,m}(\xi_{i},\xi_{j})\bigg\}_{i,j=1}^{\ell}.
\end{equation}
\end{proposition}

\begin{proof}
The event that $N^{\frac{m}{2}}\lambda^{(m)}_{\mathrm{min},N} > s$ is equivalent to the event that the interval $[0,s\,N^{-\frac{m}{2}}]$ contains no eigenvalues. The probability of the latter is expressible through standard manipulations in terms of the correlation functions of the point process (see \textit{e.g.} Chapter $5$ of \cite{DVJ88}),
\begin{align}
\mathbb{P}\left(N^{\frac{m}{2}}\lambda^{(m)}_{\mathrm{min},N} > s\right) &= \mathbb{P}\left(\{\textrm{the interval}\,\, [0,sN^{-\frac{m}{2}}] \,\,\textrm{contains no eigenvalues}\}\right)\\
&= \sum_{\ell=1}^{N}\frac{(-1)^{\ell}}{\ell!}\,\int_{[0,s\,N^{-\frac{m}{2}}]^{\ell}}\prod_{j=1}^{\ell}dx_{j}\,\mathrm{Pf}\bigg\{K_{N}(x_{i},x_{j})\bigg\}_{i,j=1}^{\ell}\\
&= \sum_{\ell=1}^{N}\frac{(-1)^{\ell}}{\ell!}\,\int_{[0,s]^{\ell}}\prod_{j=1}^{\ell}d\xi_{j}\,\mathrm{Pf}\bigg\{k^{(\mathrm{origin})}_{N}(\xi_{i},\xi_{j})\bigg\}_{i,j=1}^{\ell}
\end{align}
where we changed variable $x_{j} = N^{-\frac{m}{2}}\xi_{j}$ for each $j=1,\ldots,\ell$. Passing to the limit with the uniform convergence of Proposition \ref{prop:kernolim} and applying Lemma \ref{AGZ_lemma} as in the proof of Theorem \ref{th:lambdamax} completes the proof.
\end{proof}

\section{Uniform asymptotics of a truncated hypergeometric series}
In this Appendix we obtain the asymptotics of the sum \eqref{ginfdef1} close to the transition region $x=1$ and uniformly above this region. Similar asymptotics without uniform error bounds were given in \cite[Appendix C]{AB12}. The case $m=1$ has been studied thoroughly by several authors, we refer to the article \cite{BG07} that inspired the approach taken below for general $m$.
\begin{lemma}
\label{lem:fnedgelem}
Let $s>0$ be fixed. Then uniformly on $x>1-s/\sqrt{N}$ we have
\begin{equation}
f_{N-2}(N^{m}x) \sim \frac{e^{Nm}}{(2\pi)^{\frac{m}{2}}}\,\sqrt{\frac{\pi}{2m}}\,N^{-\frac{m-1}{2}}\,x^{N-1}e^{\frac{N(x-1)^{2}}{2m}}\,\mathrm{erfc}\left(\frac{\sqrt{N}(x-1)}{\sqrt{2m}}\right), \label{fnedge}
\end{equation}
as $N \to \infty$.
\end{lemma}
\begin{proof}
Let us assume that $x\geq 1$ to begin with. We start from the identity
\begin{equation}
\frac{N^{j}}{j!} = \frac{1}{2\pi i}\oint_{C}dz\,\frac{e^{Nz}}{z^{j+1}}
\end{equation}
where $C$ is a circle centered at $0$ with radius $r_{N} = 1-N^{-\frac{1}{2}-\delta}$ for a fixed $\delta>0$. Inserting this into the definition of $f_{N-2}(N^{m}x)$ and summing the resulting geometric series gives
\begin{equation}
f_{N-2}(N^{m}x) = \frac{1}{(2\pi i)^{m}}\oint_{C^{m}}\left(\prod_{j=1}^{m}\frac{dz_{j}}{z_{j}}\,e^{Nz_{j}}\right)\frac{\left(\frac{x}{z^{(m)}}\right)^{N-1}-1}{\frac{x}{z^{(m)}}-1}, \label{applycauchy}
\end{equation}
where $z^{(m)} = z_{1}z_{2}\ldots z_{m}$. Since $x \geq 1 > r_{N}$, one of the terms in \eqref{applycauchy} is analytic inside the closed integration contour. By Cauchy's theorem this contribution is equal to zero. We thus obtain the identity
\begin{equation}
f_{N-2}(N^{m}x) = \frac{x^{N-1}}{(2\pi i)^{m}}\,\oint_{C^{m}}\left(\prod_{j=1}^{m}dz_{j}\,e^{N\phi(z_j)}z_{j}\right)\,\frac{1}{x-z^{(m)}}, \label{erfcderivident}
\end{equation}
where $\phi(z) = z-\log(z)$. Parameterising the contour $C$ by $z = r_{N}e^{i\theta}$ with $\theta \in [-\pi,\pi)$, the function $\phi(z)$ takes the form
\begin{equation}
\phi(r_{N}e^{i\theta}) = r_{N}e^{i\theta}-i\theta-\log(r_{N}) \label{actionfnedge}
\end{equation}
with $\mathrm{Re}(\phi(z)) = r_{N}\cos(\theta)-\log(r_{N})$. Thus $\mathrm{Re}(\phi(z))$ is maximised when $\theta=0$. This combined with the simple bound $|x-z^{(m)}| \geq CN^{-\frac{1}{2}-\delta}$ shows that the contribution to \eqref{erfcderivident} from the complement of the intervals $[-\epsilon,\epsilon]^{m}$ is exponentially suppressed for any $\epsilon>0$ and can be neglected.

Taking $\epsilon>0$ sufficiently small, we Taylor expand the action \eqref{actionfnedge} near $\theta=0$,
\begin{equation}
\phi(z) = 1+i\theta(r_{N}-1)-\frac{\theta^{2}}{2}+O(\theta^{3})+O(N^{-1-\delta}) + O(\theta^{2}N^{-\frac{1}{2}-\delta})
\end{equation}
where we used that $r_{N}-1-\log(r_{N}) = O(N^{-1-\delta})$ and $r_{N}(e^{i\theta}-1-i\theta) = O(\theta^{3})$ as $\theta \to 0$, uniformly in $N$. Next consider the denominator term in \eqref{erfcderivident}. We have
\begin{equation}
\mathrm{Re}(x-z^{(m)}) = x-r_{N}^{m}\cos(\theta_1+\ldots+\theta_m) \geq CN^{-\frac{1}{2}-\delta} > 0\label{repartbnd}
\end{equation}
and we can write
\begin{equation}
\frac{1}{\sqrt{N}(x-z^{(m)})}= \int_{0}^{\infty}dt\,e^{-t\sqrt{N}(x-z^{(m)})}. \label{expintrep}
\end{equation}
Choose a new parameter $\tilde{\delta}$ such that $\tilde{\delta} > \delta > 0$. Then by \eqref{repartbnd} we have that uniformly on $x \geq 1$,
\begin{equation}
\int_{N^{\tilde{\delta}}}^{\infty}dt\,e^{-t\sqrt{N}(x-z^{(m)})} = O(e^{-cN^{\tilde{\delta}-\delta}}), \qquad N \to \infty,
\end{equation}
and therefore we can neglect the part of the $dt$ integral with $t>N^{\tilde{\delta}}$. Now we apply Fubini's theorem to interchange the $d\theta$ and $dt$ integrals, split $-t\sqrt{N}(x-z^{(m)}) = -t\sqrt{N}(x-r_{N}^{m}) +  t\sqrt{N}(z^{(m)}-r_{N}^{m})$ and Taylor expand 
\begin{align}
z^{(m)}-r_{N}^{m} &= r_{N}^{m}(e^{i(\theta_1+\ldots+\theta_m)}-1)\\
&= r_{N}^{m}i(\theta_1+\ldots+\theta_{m}) + O(\vec{\theta}^{2}).
\end{align}
Inserting these results we obtain
\begin{align}
&f_{N-2}(N^{m}x) = \frac{x^{N-1}e^{Nm}\sqrt{N}}{(2\pi)^{m}}\int_{0}^{N^{\tilde{\delta}}}dt\,e^{-t\sqrt{N}(x-r_{N}^{m})}\\
&\times \int_{[-\epsilon,\epsilon]^{m}}\,\left(\prod_{j=1}^{m}d\theta_{j}\,e^{-\frac{N}{2}\theta_{j}^{2}+iN(r_{N}-1)\theta_{j}+it\sqrt{N}r_{N}^{m}\theta_{j}}\right)\,e^{NO(\vec{\theta}^{3})+N^{\frac{1}{2}-\delta}O(\vec{\theta}^{2})+O(t\sqrt{N}\vec{\theta}^{2})}\,(1+O(N^{-\delta})), \label{bigoexp}
\end{align}
where all of the big-$O$ error terms are uniform. These big-$O$ terms can be neglected as follows: Since $t < N^{\tilde{\delta}}$ we have $t\sqrt{N}|\vec{\theta}|^{2} < cN|\vec{\theta}|^{2}$ for a $c>0$ we can take as small as desired. Similarly $N\tilde{\phi}(\theta) < CN|\vec{\theta}|^{3} < cN|\vec{\theta}|^{2}$. Then using the bound $|e^{z}-1| \leq |z|e^{|z|}$ for all $z \in \mathbb{C}$ we get
\begin{align}
&\bigg{|}\sqrt{N}\int_{0}^{N^{\tilde{\delta}}}dt\,e^{-t\sqrt{N}(x-r_{N}^{m})}\\
&\times\int_{[-\epsilon,\epsilon]^{m}}\,\prod_{j=1}^{m}d\theta_{j}\,e^{-\frac{N}{2}\theta_{j}^{2}+iN(r_{N}-1)\theta_{j}+it\sqrt{N}r_{N}^{m}\theta_{j}}(e^{NO(\vec{\theta}^{3})+N^{\frac{1}{2}-\delta}O(\vec{\theta}^{2})+O(t\sqrt{N}\vec{\theta}^{2})}-1)\bigg{|}\\
&\leq \sqrt{N}\int_{0}^{N^{\tilde{\delta}}}dt\,\int_{\mathbb{R}^{m}}\,\prod_{j=1}^{m}d\theta_{j}\,e^{-N(\frac{1}{2}-c)\theta_{j}^{2}}|NO(\vec{\theta}^{3})+O(t\sqrt{N}\vec{\theta}^{2})+N^{\frac{1}{2}-\delta}O(\vec{\theta}^{2})|\\
&= O(N^{-\frac{m}{2}+2\tilde{\delta}})
\end{align}
uniformly on $x \geq 1$. Choosing $\tilde{\delta}, \delta$ small enough we can neglect this contribution as the power of $N$ is strictly less than the power $N^{-\frac{m-1}{2}}$ in \eqref{fnedge}.

Next we replace the integration domain $[-\epsilon,\epsilon]^{m}$ in \eqref{bigoexp} with $\mathbb{R}^{m}$ admitting at most exponentially small errors. The resulting integral is the Fourier transform of a Gaussian:
\begin{align}
\int_{\mathbb{R}^{m}}\,\left(\prod_{j=1}^{m}d\theta_{j}\,e^{-\frac{N}{2}\theta_{j}^{2}+iN(r_{N}-1)\theta_{j}+it\sqrt{N}r_{N}^{m}\theta_{j}}\right) = \left(\frac{2\pi}{N}\right)^{\frac{m}{2}}\,e^{-m\frac{(\sqrt{N}(r_{N}-1)+tr_{N}^{m})^{2}}{2}}.
\end{align}
After inserting this identity we can now reinstate the part of the $t$-integral on $[N^{\delta},\infty)$ costing a multiplicative $(1+O(e^{-N^{\delta}}))$ error. Summarising, the preceding estimates show that uniformly on $x \geq 1$, we have
\begin{equation}
f_{N-2}(N^{m}x) =  \frac{x^{N-1}e^{Nm}}{(2\pi)^{\frac{m}{2}}N^{\frac{m-1}{2}}}\int_{0}^{\infty}dt\,e^{-t\sqrt{N}(x-r_{N}^{m})-m\sqrt{N}(r_{N}-1)tr_{N}^{m}-\frac{m}{2}t^{2}r_{N}^{2m}}(1+O(N^{-\delta})+O(N^{-\frac{1}{2}+2\tilde{\delta}}))\label{finaltint}
\end{equation}
We have used that $\frac{Nm(r_{N}-1)^{2}}{2} = O(N^{-\delta})$ and likewise, all other instances of $r_{N}$ in \eqref{finaltint} can be replaced with $1$ at the expense of a multiplicative $(1+O(N^{-\delta}))$ error. This gives
\begin{equation}
f_{N-2}(N^{m}x) \sim  \frac{x^{N-1}e^{Nm}}{(2\pi)^{\frac{m}{2}}N^{\frac{m-1}{2}}}\int_{0}^{\infty}dt\,e^{-t\sqrt{N}(x-1)-m\frac{t^{2}}{2}} 
\end{equation}
which is equivalent to the claimed formula \eqref{fnedge}.

If $1-s/\sqrt{N} < x\leq1$ the approach is similar, the only difference with the case $x \geq 1$ is that we take the radius $r_{N}$ of the circle $C$ to be $r_{N} = 1+N^{-\frac{1}{2}-\delta}$, so that now the singularity at $z=x$ is strictly inside $C$. Now we get a contribution from the first term in \eqref{applycauchy} which is the quantity $f_{\infty}(N^{m}x)$. Following all the same steps and changing the sign in the exponent of \eqref{expintrep}, we get 
\begin{align}
f_{N-2}(N^{m}x) &= f_{\infty}(N^{m}x) \label{finterf}\\
&-\frac{e^{Nm}}{(2\pi)^{\frac{m}{2}}}\,\sqrt{\frac{\pi}{2m}}\,N^{-\frac{m-1}{2}}\,x^{N-1}e^{\frac{N(x-1)^{2}}{2m}}\,\mathrm{erfc}\left(\frac{-\sqrt{N}(x-1)}{\sqrt{2m}}\right)\,(1+o(1)) \label{finterft2}\\
&= \frac{e^{Nm}}{(2\pi)^{\frac{m}{2}}}\,\sqrt{\frac{\pi}{2m}}\,N^{-\frac{m-1}{2}}\,x^{N-1}e^{\frac{N(x-1)^{2}}{2m}}\,\mathrm{erfc}\left(\frac{\sqrt{N}(x-1)}{\sqrt{2m}}\right)\,(1+o(1)) \label{fintfinal}.
\end{align}
To obtain the last estimate \eqref{fintfinal} we inserted the asymptotics of $f_{\infty}(N^{m}x)$ into \eqref{finterf} from \eqref{asyfint}. Using the assumption that $1-s/\sqrt{N} \leq x < 1$ shows that both terms \eqref{finterf} and \eqref{finterft2} are of the same order in $N$. Finally, the identity $1-\frac{1}{2}\mathrm{erfc}(-x) = \frac{1}{2}\mathrm{erfc}(x)$ is used and one obtains \eqref{fintfinal}.
\end{proof}

\section{$L^{1}$ estimate and miscellaneous bounds}
\label{se:truncation}
Let us denote for some small positive constant $c>0$,
\begin{equation}
\tau(f) := \sup_{x \in \mathbb{R}}\bigg\{\lvert f(x) \rvert e^{-cx^{\frac{2}{m}}}\bigg\}.
\end{equation}
\begin{lemma}
\label{le:l1est}
Let $f$ satisfy the assumptions of Theorem \ref{th:ginconv}, in particular that $\tau(f) < \infty$. Then the complementary statistic $\tilde{\xi}^{\mathrm{c}}_{N,m}(f)$ defined in \eqref{xidef} satisfies the $L^{1}$ estimate,
\begin{equation}
\mathbb{E}\left(\bigg{|}\tilde{\xi}^{\mathrm{c}}_{N,m}(f)\bigg{|}\right) = O(N^{\epsilon}), \label{l1est}
\end{equation}
as $N \to \infty$. In the mesoscopic regime we have the same estimate, for any $E \in (-1,1)$ and $\tau>0$,
\begin{equation}
\mathbb{E}\left(\bigg{|}\tilde{\xi}^{(\tau),\mathrm{c}}_{N,m}(f)\bigg{|}\right) = O(N^{\epsilon}). \label{l1estmeso}
\end{equation}
\end{lemma}

\begin{proof}
The proofs of \eqref{l1est} and \eqref{l1estmeso} are almost identical, so we focus on \eqref{l1est}. We start from the identity,
\begin{equation}
\mathbb{E}\left(\sum_{j=1}^{n}|f(\lambda_j)|\mathbbm{1}_{\lambda_{j} \in \mathcal{E}_{N}^{\mathrm{c}}}\right)  = \int_{\mathcal{E}_{N}^{\mathrm{c}}}dx\, |f(x)|S_{N}(x,x), \label{explinstat}
\end{equation}
where $S_{N}(x,x)$ is the scalar kernel \eqref{prekernel} of Theorem \ref{th:ik} evaluated on the diagonal. We shall make use of the alternative representation \eqref{Srep1}. Without loss of generality we suppose that $f$ is supported on the half-line $[0,\infty)$, if it is not we can apply the same steps owing to the symmetry $S_{N}(x,x) = S_{N}(-x,-x)$. The set $\mathcal{E}_{N}^{\mathrm{c}}$ consists of two parts: the outer-edge region where $x \geq \alpha_{N} := 1-N^{-\frac{1}{2}+\epsilon}$, and the origin region where $0 < x  < N^{-\frac{m}{2}+\epsilon}$. We begin with the outer-edge region. 

Inserting \eqref{Srep1}, we find \begin{align}
&\int_{\alpha_{N}}^{\infty}dx\,|f(x)|S_{N}(x,x) \leq 2C_{N,m}\int_{\alpha_N}^{\infty}dx\,|f(x)|\int_{x}^{\infty}dy\,(y-x)\,w(N^{\frac{m}{2}}x)w(N^{\frac{m}{2}}y)f_{N-2}(N^{m}xy)\label{iout1}\\
&+C_{N,m}N^{\frac{m(N-3)}{2}}2^{m\left(\frac{N-1}{2}\right)}\left(\frac{\Gamma\left(\frac{N-1}{2}\right)}{(N-2)!}\right)^{m}\int_{\alpha_N}^{\infty}dx\,|f(x)|x^{N-1}w(N^{\frac{m}{2}}x). \label{iout2}
\end{align}
We denote by $I_{1}$ the double integral term in \eqref{iout1} and $I_{2}$ the expression in \eqref{iout2}. We further denote by $I_{1,\mathrm{ext}}$ the double integral term in \eqref{iout1} with the integration over $x$ restricted to $x \geq 1$. Proposition \ref{prop:gin1} and Lemma \ref{lem:fnedgelem} provide the following bounds, uniform on $x, y \in [1,\infty)$:
\begin{align}
w(N^{\frac{m}{2}}x) &\leq CN^{-\frac{m-1}{2}}e^{-\frac{Nm}{2}x^{\frac{2}{m}}}, \label{wbndl1}\\
f_{N-2}(N^{m}xy) & \leq Ce^{Nm}N^{-\frac{m-1}{2}}(xy)^{N-1}. \label{fbndl1}
\end{align}
Then with $\phi(x) = \frac{m}{2}(x^{\frac{2}{m}}-1-\log(x))$ and recalling that $C_{N,m} = O(N^{\frac{3m}{2}})$ we get
\begin{equation}
|I_{1,\mathrm{ext}}| \leq CN^{\frac{3}{2}}\int_{1}^{\infty}dx\,|f(x)|\int_{x}^{\infty}dy\,(y-x)\,e^{-N\phi(x)-N\phi(y)}. \label{i1ext}
\end{equation}
If $x>1+\epsilon$ or $y>1+\epsilon$, the latter is bounded by 
\begin{align}
&CN^{\frac{3}{2}}e^{-2(N-1)\phi(1+\epsilon)}\int_{1}^{\infty}dx\,|f(x)|e^{-\phi(x)} \\
&\leq CN^{\frac{3}{2}}e^{-2(N-1)\phi(1+\epsilon)}\tau(f)\int_{1}^{\infty}xe^{-(\frac{m}{2}-c)x^{\frac{2}{m}}}\\
& \leq Ce^{-c_{\epsilon}N}. \label{firstl1bnd}
\end{align}
If $1 \leq x < 1+\epsilon$ and $y<1+\epsilon$ in \eqref{i1ext} we apply the the Laplace method and Taylor expand $\phi(x)$ and $\phi(y)$ at the critical point $(x,y)=(1,1)$ picking up two factors of $N^{-\frac{1}{2}}$ and a third $N^{-\frac{1}{2}}$ from the factor of $(y-x)$. We thus obtain that $|I_{1,\mathrm{ext}}| \leq C\,\tau(f)$. For $I_{2}$ we use Stirling's formula to see that the pre-factors in \eqref{iout2} satisfy
\begin{equation}
C_{N,m}\left(\frac{\Gamma\left(\frac{N-1}{2}\right)}{(N-2)!}\,N^{\frac{N}{2}}\,2^{\frac{N-1}{2}}\right)^{m} \sim (2N)^{\frac{m}{2}}\,e^{\frac{Nm}{2}}, \qquad N\to \infty.
\end{equation}
Combined with the bound \eqref{wbndl1} we deduce that \begin{equation}
|I_{2}| \leq C\,\sqrt{N}\int_{\alpha_{N}}^{\infty}dx\,|f(x)|e^{-N\phi(x)}. \label{i2bnd}
\end{equation}
The contribution to the integral \eqref{i2bnd} on $x > 1+\epsilon$ is bounded by $\tau(f)e^{-c_{\epsilon}N}$ by the same approach used to prove \eqref{firstl1bnd}, so we restrict the range of integration to $x \in [1,1+\epsilon]$. Then we can Taylor expand $\phi(x)$ near $x=1$ and the standard techniques of the Laplace method gives us $|I_{2}| \leq C\,\tau(f)$.
 
It remains to control the contribution to $I_{1}$ from the integration range $x \in (1-N^{-\frac{1}{2}+\epsilon},1]$ that we denote $I_{1,\mathrm{int}}$. If $y>1+\epsilon$, the same strategy above gives an exponentially small contribution, so we shall assume that $y < 1+\epsilon$ in what follows. We again apply \eqref{wbndl1}, but now we use $f_{N-2}(N^{m}xy) \leq f_{\infty}(N^{m}xy)$ and note that $f_{\infty}(N^{m}xy) \leq N^{-\frac{m-1}{2}}e^{Nm(xy)^{\frac{1}{m}}}$ from Proposition \ref{prop:gin3}. Combining these bounds, we have 
that 
\begin{equation}
|I_{1,\mathrm{int}}| \leq \tau(f)\,N^{\frac{3}{2}}\,\int_{1-N^{-\frac{1}{2}+\epsilon}}^{1}dx\,\int_{x}^{1+\epsilon}dy\,(y-x)\,\mathrm{exp}\left(-\frac{Nm}{2}\left(x^{\frac{1}{m}}-y^{\frac{1}{m}}\right)^{2}\right). \label{Iintbnd}
\end{equation}
Taylor expanding near $y=x$ the latter is bounded by a constant times
\begin{align}
&\tau(f)\,N^{\frac{3}{2}}\,\int_{1-N^{-\frac{1}{2}+\epsilon}}^{1}dx\,\int_{x}^{1+\epsilon}dy\,(y-x)\,e^{-\frac{Nm}{2}(y-x)^{2}}\\
&\sim \tau(f)\,N^{\epsilon} \label{finall1bnd}
\end{align}
where the last asymptotic follows from changing variable $y \to x+y/\sqrt{N}$ and applying the dominated convergence theorem.

Next we deal with the origin region, namely the integral \eqref{explinstat} with the integration range restricted to $0 < x <  N^{-\frac{m}{2}+\epsilon}$. Note that in this regime the second term in \eqref{Srep1} is exponentially small due to the factor $x^{N-1}$ and can be neglected. For the first term in \eqref{Srep1} we consider a large constant $M>0$ and consider the contribution from the interval $x \in [0,M\,N^{-\frac{m}{2}}]$. On this interval, we simply change variable $x \to x/N^{\frac{m}{2}}$ and $y \to y/N^{\frac{m}{2}}$ which exactly cancels the $O(N^{\frac{3m}{2}})$ growth of the pre-factor $C_{N,m}$. We therefore just need to show the following quantity is finite:
\begin{align}
I_{\mathrm{origin}} &:= \tau(f)\int_{0}^{M}dx\,\int_{0}^{\infty}dy\,(x+y)\,w(x)w(y)f_{\infty}(xy)\\
&\leq \tau(f)\int_{0}^{M}w(x)dx\,\sum_{k=0}^{\infty}\frac{x^{k}}{(k!)^{m}}\int_{0}^{\infty}dy\,y^{k}w(y)\\
&= \frac{1}{2}\tau(f)\int_{0}^{M}w(x)dx\,\sum_{k=0}^{\infty}\frac{x^{k}}{(k!)^{m}}\,\Gamma\left(\frac{k+1}{2}\right)\,2^{-m\frac{k+1}{2}}, \label{origincntrl}
\end{align}
where we used \eqref{eqmomweight}. The infinite series in \eqref{origincntrl} has infinite radius of convergence. This implies 
\begin{equation}
I_{\mathrm{origin}} \leq C\,\tau(f)\int_{0}^{M}w(x)dx < C'\,\tau(f),
\end{equation}
as desired. The remaining interval $x \in [MN^{-\frac{m}{2}},N^{-\frac{m}{2}+\epsilon}]$ will turn out to give the main contribution. We use the bound $f_{N-2}(N^{m}xy) \leq f_{\infty}(N^{m}xy)$ and Propositions \ref{prop:gin1} and \ref{prop:gin3}. Further changing variables $x=u^{m}$, $y=v^{m}$ we obtain
\begin{equation}
\int_{MN^{-\frac{m}{2}}}^{N^{-\frac{m}{2}+\epsilon}}dx\,|f(x)|S_{N}(x,x) \leq N^{\frac{3}{2}}\tau(f)\int_{M^{\frac{1}{m}}N^{-\frac{1}{2}}}^{N^{-\frac{1}{2}+\frac{\epsilon}{m}}}du\,\int_{u}^{u+\epsilon}dv\,\frac{v^{m}-u^{m}}{(uv)^{\frac{3m}{2}}}\,e^{-\frac{Nm}{2}(u-v)^{2}} \label{obnd}
\end{equation}
Now changing variable $v \to u+v/\sqrt{N}$ we express the integrand in terms of the function $Q(u,v/\sqrt{N})$ of Lemma \ref{lem:qbound} so that the right-hand side of \eqref{obnd} is bounded by
\begin{align}
&\tau(f)N^{\frac{1}{2}}\int_{M'N^{-\frac{1}{2}}}^{N^{-\frac{1}{2}+\frac{\epsilon}{m}}}du\,\int_{0}^{\sqrt{N}\epsilon}dv\,\sqrt{N}Q\left(u,\frac{v}{\sqrt{N}}\right)\,e^{-\frac{m}{2}v^{2}}\\
&\leq C\,\tau(f)N^{\frac{\epsilon}{m}}	
\end{align}
where the boundedness of the function $\sqrt{N}Q\left(u,\frac{v}{\sqrt{N}}\right)$ provided $u > N^{-\frac{1}{2}}$ follows from Lemma \ref{lem:qbound}. 
\end{proof}
We have the following corollaries. Let $J_{\mathrm{edge}} = (1-N^{-\frac{1}{2}+\epsilon},\infty)$ and $J_{\mathrm{origin}} = [0,N^{-\frac{m}{2}+\epsilon}]$.
\begin{corollary}
For a product $X = N^{-\frac{m}{2}}G_{1}G_{2}\ldots G_{m}$ of $m$ independent real Ginibre matrices, let $N_{J}$ denote the number of real eigenvalues of $X$ lying inside the set $J$. Then
\begin{align}
\mathbb{E}(N_{J_{\mathrm{edge}}}) &= O(N^{\epsilon}),\\
\mathbb{E}(N_{J_{\mathrm{origin}}}) &= O(N^{\frac{\epsilon}{m}}).
\end{align}
\end{corollary}
The following Lemma is given in \cite{LMS21}. For completeness we repeat the proof here.
\begin{lemma}
\label{lem:qbound}
Consider the function $Q : [0,\infty)^{2} \to [0,\infty)$ defined by
\begin{equation}
Q(u,v) = \frac{(u+v)^{m}-u^{m}}{(u(u+v))^{\frac{m-1}{2}}}. \label{qdef}
\end{equation}
Then for $\delta>0$ sufficiently small and $M>0$ arbitrary, we have for $(u,v) \in [0,M] \times [0,\delta]$ the uniform bound
\begin{equation}
Q(u,v) = mv + \sum_{j=0}^{m-1}O\left(\frac{v^{3+j}}{u^{2+j}}\right). \label{qsmallu}
\end{equation}
Furthermore, for any $\epsilon>0$, there is a constant $C_{\epsilon,M}>0$ independent of $u$ such that on the domain $(u,v) \in [\epsilon,M] \times [0,M]$ we have $Q(u,v) \leq C_{\epsilon,M}v$.
\end{lemma}

\begin{proof}
We have
\begin{equation}
Q(u,v) = u\frac{\left(1+\frac{v}{u}\right)^{m}-1}{\left(1+\frac{v}{u}\right)^{\frac{m-1}{2}}} = \left(1+\frac{v}{u}\right)^{-\frac{m-1}{2}}\sum_{j=0}^{m-1}\binom{m}{j+1}\frac{v^{j+1}}{u^{j}}. \label{qbin}
\end{equation}
This implies that $Q(u,v)\leq C_{\epsilon,M}v$ provided $u > \epsilon$. Taylor expanding near $v=0$, we have
\begin{equation}
\left(1+\frac{v}{u}\right)^{-\frac{m-1}{2}} = 1 - \frac{m-1}{2}\,\frac{v}{u} + O\left(\frac{v^{2}}{u^{2}}\right), \qquad 0 < v < \delta. \label{unifbigo}
\end{equation}
To obtain the uniform big-$O$ term in \eqref{unifbigo} note that 
\begin{equation}
\frac{d^{2}}{dv^{2}}\left(1+\frac{v}{u}\right)^{-\frac{m-1}{2}} = \frac{\left(1+\frac{v}{u}\right)^{-\frac{m-1}{2}}}{4(u+v)^{2}}(m^{2}-1) \leq \frac{m^{2}-1}{4u^{2}}.
\end{equation}
Then inserting \eqref{unifbigo} into \eqref{qbin} the term proportional to $\frac{v^{2}}{u}$ cancels and we obtain \eqref{qsmallu}.
\end{proof}

\bibliographystyle{alpha}
\bibliography{productsbib}

\end{document}